\documentclass[preprint, 11pt]{elsarticle}
\usepackage{amsmath}
\usepackage{amssymb}
\usepackage{amsthm}
\usepackage{amscd}
\usepackage{todonotes} 

\usepackage{color}

\def\m{{\mathfrak m}} 

\def\p{{\mathfrak p}}    
\def\q{{\mathfrak q}}    
 
\def\N{{\mathfrak N}}

\DeclareMathOperator{\Spec}{Spec}

\begin{document}

\newtheorem{theorem}{Theorem}[section]
\newtheorem{lemma}[theorem]{Lemma}
\newtheorem{proposition}[theorem]{Proposition}
\newtheorem{corollary}[theorem]{Corollary}
\newtheorem{problem}[theorem]{Problem}
\newtheorem{question}[theorem]{Question}
\newtheorem{notation}[theorem]{Notation}
\newtheorem{setting}[theorem]{Setting}

\newtheorem{example}[theorem]{Example}
\newtheorem{defi}[theorem]{Definitions}
\newtheorem{definition}[theorem]{Definition}
\newtheorem{remark}[theorem]{Remark}
\newtheorem{ex}[theorem]{Example}
\renewcommand{\thedefi}{}

\long\def\alert#1{\smallskip{\hskip\parindent\vrule%
\vbox{\advance\hsize-2\parindent\hrule\smallskip\parindent.4\parindent%
\narrower\noindent#1\smallskip\hrule}\vrule\hfill}\smallskip}

\def\ff{\frak}
\def\Proj{\mbox{\rm Proj}}
\def\type{\mbox{ type}}
\def\Hom{\mbox{ Hom}}
\def\rank{\mbox{ rank}}
\def\Ext{\mbox{ Ext}}
\def\Ker{\mbox{ Ker}}
\def\Max{\mbox{\rm Max}}
\def\Min{\mbox{\rm Min}}
\def\End{\mbox{\rm End}}
\def\xpd{\mbox{\rm xpd}}
\def\Ass{\mbox{\rm Ass}}
\def\emdim{\mbox{\rm emdim}}
\def\epd{\mbox{\rm epd}}
\def\repd{\mbox{\rm rpd}}
\def\ord{{\rm ord}}
\def\hgt{{\rm ht}}

\def\edeg{{\rm e}}

\begin{frontmatter}

\title{Ideal theory of 
infinite directed unions of local \\quadratic transforms}
\author[Bill]{William Heinzer}
\ead{heinzer@purdue.edu}
\address[Bill]{Department of Mathematics, Purdue University, West Lafayette, Indiana 47907}
\author[Alan]{K. Alan Loper}
\ead{lopera@math.ohio-state.edu}
\address[Alan]{Department of Mathematics, Ohio State
University - Newark, Newark, OH 43055}
\author[Bruce]{Bruce Olberding}
\ead{olberdin@nmsu.edu}
\address[Bruce]{Department of Mathematical Sciences, New Mexico State University, Las Cruces, NM 88003-8001}
\author[Hans]{Hans Schoutens}
\ead{hschoutens@citytech.cuny.edu}
\address[Hans]{Department of Mathematics, 365 5th Avenue, The CUNY Graduate Center, New York, NY 10016 USA} 
\author[Matt]{Matthew Toeniskoetter}
\ead{mtoenisk@purdue.edu}
\address[Matt]{Department of Mathematics, Purdue University, West Lafayette, Indiana 47907}

\begin{abstract}
Let $(R, \m)$ be a  regular local ring of dimension at least 2.  
Associated to each  valuation domain birationally dominating $R$, there exists 
a unique sequence  $\{R_n\}$ of local quadratic transforms of $R$ along this valuation domain.
We consider the situation where the sequence $\{ R_n \}_{n \ge 0}$ is infinite, and 
examine ideal-theoretic properties of the integrally closed local domain  $S = \bigcup_{n \ge 0} R_n$.
Among the set of valuation overrings of $R$,  there exists a 
 unique limit point  $V$  for the sequence of order valuation rings of the $R_n$.  We prove the 
 existence of  a unique minimal proper Noetherian overring $T$ of $S$, and establish the 
 decomposition   $S = T \cap V$.     
 If $S$ is archimedian, then the complete integral closure $S^{*}$ of $S$ 
 has the form $S^{*} = W \cap T$, where $W$ is the rank $1$  valuation overring of $V$.
\end{abstract}

\begin{keyword}
Regular local ring \sep 
local quadratic transform \sep valuation ring \sep complete integral closure 
\MSC[2010]  13H05  \sep 13A15  \sep 13A18 
\end{keyword}

\end{frontmatter}

\section{Introduction}


Let $(R,\m)$ be a $d$-dimensional regular local ring with $d \geq 2$. The morphism $\phi : {\rm Proj} \: R[\m t] \rightarrow \Spec R$ defines the blow-up of the maximal ideal $\m$ of $R$. Let $(R_1,\m_1)$ be the local ring of a point in the fiber of $\m$ defined by $\phi$. Then $R_1$, called a {\it local quadratic transform} of $R$, is a regular local ring of dimension at most $d$ that birationally dominates $R$.  Local quadratic transforms have historically played an important role in resolution of singularities and  in the understanding of regular local rings. Classically, Zariski's unique factorization theorem for ideals in a 2-dimensional regular local ring  \cite{Zar38} relies on local quadratic transforms in a fundamental way. The $2$-dimensional regular local rings birationally dominating $R$ are all iterated local quadratic transforms of $R$, and they are in one-to-one correspondence with the simple complete $\m$-primary ideals of $R$. More recently, Lipman \cite{Lip}  uses similar methods to prove a unique factorization theorem for a special class of complete ideals in regular local rings of dimension $\ge 2$.

By taking regular local rings of dimension at least 2, iteration of the process of local quadratic transforms yields an
infinite sequence $\{(R_n,\m_n)\}_{n\geq 0}$. We consider the directed union of this infinite sequence of local quadratic
transforms, and set $S = \bigcup_{n\geq 0}R_n$. Since the rings $R_n$ are local rings that are linearly ordered under domination, 
$S$ is local, and since the rings $R_n$ are integrally closed, so is $S$. However, if $S$ is not a discrete valuation ring, then $S$ is not Noetherian.
On the other hand, one may consider a valuation ring $(V, {\mathfrak{N}} )$ that dominates $R$. There is a unique local quadratic transform $R_1$ of $R$ that is dominated by $V$, called the {\it local quadratic transform of R along $V$}. If $V$ is the order valuation ring of $R$, then $R_1 = V$, but otherwise, one may take the local quadratic transform $R_2$ of $R_1$ along $V$. Specifically, $R_1 =   R[\m/x]_{\N \cap R[\m/x]}$, where $x \in \m$ is such that $x V = \m V$.

 Iterating this process yields a possibly infinite sequence $\{(R_n,\m_n)\}$ of local quadratic transforms, where this process terminates if and only if $V$ is the order valuation ring of some $R_n$. Abhyankar \cite[Proposition 4]{Abh} proves   that this sequence is finite if and only if the transcendence degree of $V/{\ff N}$ over $R/\m$ is $d-1$ (that is, by the dimension formula \cite[Theorem~15.5]{Mat}, the residual transcendence degree of $V/{\ff N}$ is as large as possible; in such a case $V$ is said to be a {\it prime divisor} of $R$). Otherwise, the induced sequence is infinite, and it is in this case that we are especially interested in this article.


  Abhyankar  \cite[Lemma~12]{Abh}  proves that if $\dim R = 2$, then the union  $\bigcup_{i \geq 0}R_i$ is equal to $V$. For the setting where $\dim R>2$, Shannon \cite{Sha} presents several examples showing that  the directed union $S =  \bigcup_{i \geq 0}R_i $ is properly contained in $V$, and in particular $S$ is not a valuation ring.    
 More generally, Lipman   \cite[Lemma 1.21.1]{Lip}  observes that   if $P$ is a nonmaximal prime ideal of 
 the  regular local ring $R$,   then there exists an  infinite sequence of local quadratic transforms of $R$
  whose union $S$  is contained in $R_P$.  Thus if  we take $P$ so that  $\hgt \:P > 1$,  then 
  $S$ cannot be a valuation ring, since the overring $R_P$ of $S$ is not a valuation ring.   
  Thus arises the question of the nature of $S$ when $S$ is not a valuation ring.

    Since Shannon's work in \cite{Sha}  has sparked the present authors' interest in this topic,  
we refer to 
 the directed union $ S =\bigcup_{i \geq 0}R_i $  of the local quadratic transforms of $R$ along a valuation ring 
  as   the {\it Shannon extension} of
 $R$ along this valuation ring. In this article we examine the nature of Shannon extensions, with special emphasis on the  ideal theory and representation  of such rings. We prove  in Theorem~\ref{flat} that if $S$ is a Shannon extension, 
  then  there exists a unique minimal proper Noetherian overring $T$ of $S$    and 
a valuation overring  $V$ of $R$   such 
 that $S = T \cap V$. The ring $T$ is even a localization of one of the $R_i$ and hence is itself a regular Noetherian domain (Theorem~\ref{hull}).
 The valuation domain $V$, which we term the boundary valuation ring of $S$, is the unique limit point in the patch topology of  the order valuation rings of the regular local rings $R_n$ (Corollary~\ref{equation V cor}).  While $V$  determines the sequence of the $R_i$'s, it is not generally unique in doing so.
 In fact, if $S$ is not a valuation ring, then there are infinitely many valuation rings that give rise to the same Shannon extension.  However, in light of its topological interpretation, the boundary valuation ring is in a sense the valuation ring that is ``preferred'' by the sequence. 
  
 From this representation we deduce in Corollary~\ref{chain}
 that the principal $N$-primary ideals of $S$ are linearly ordered with respect to inclusion.
 In Theorem~\ref{cic arch},  we use the representation of $S$ to  describe the complete integral closure $S^*$ of $S$ 
 in the case where $S$ is archimedean,  and in Theorem~\ref{cic},  we describe $S^*$ in
 the case where $S$ is not archimedean. 
 Along the way in Sections 3 and 4
 we also describe the maximal and nonmaximal prime ideals of Shannon extensions. To illustrate several of the ideas in the paper, we present details in Examples~\ref{Shannon example 4.7} and~\ref{Shannon example} about an example given by David Shannon that motivated our work in this paper.  For these examples, we explicitly describe the Noetherian hull  and the  complete integral closure  of the Shannon extension $S$. 
 
 As we show in Section 8, our representation of a Shannon extension $S$ also proves useful in determining when  $S$ is a valuation ring, since in terms of our representation theorem, this is equivalent to asking when $S$ is equal to its boundary valuation ring. 
   %
  %
%
Shannon proves in  \cite[Prop.~ 4.18]{Sha}  that if $S$ is a Shannon extension $S = \bigcup_{i}R_i$ along a valuation ring $V$ that is nondiscrete and rank $1$, then   $S = V$ if and only if for every height $1$ prime ideal $P$ of any $R_i$, we have $(R_i)_{P} \not\supseteq S$.  If this condition  holds,  Shannon says the sequence $\{R_i\}$
 {\it  switches strongly infinitely often}.   
 More recently,   in  \cite{Gra}, \cite{GMR}, \cite{GMR2}, \cite{GR},  \cite{Gra2},   
   interesting work has been done by A.  Granja,  M. C. Martinez,  C. Rodriquez and T. S\'anchez-Giralda, 
    on the question of when $S=V$.     In \cite[Theorem~13]{Gra},  Granja proves that    $S =  V$ if 
and only if either 
\begin{enumerate}
\item  the sequence $\{R_i\}$ switches strongly infinitely often,  or   
\item   there exists  a unique rank  one valuation domain $W$ such that $W$ is a localization of $R_n$ for all
large $n$.
\end{enumerate} 
In the case of item~1,  the valuation ring  $V$  has rank 1,  while in the case of item~2, 
the valuation ring $V$ has rank 2 and $V$ is contained in $W$.  In the case of item~2,  
the sequence 
$\{R_i\}$  is  said to be 
{\it height 1 directed}.   This describes the fact that $W$ is a localization of $R_n$  for all large $n$.

In this same direction, 
in Theorem~\ref{valuation criterion},  we 
 prove that $S$ is a valuation ring $\iff$   $S$ has only finitely many height $1$ 
 prime ideals  $\iff$ either (a)  $\dim S = 1$ or (b) $\dim S = 2$,  and the boundary valuation
 ring $V$ of $S$ has value group $\mathbb Z \oplus G$ ordered lexicographically,  where
 $G$ is an ordered subgroup of $\mathbb Q$.  We also show how to recover some of the results of Abhyankar and Granja from our point of view.

In general,  our notation is as in Matsumura   \cite{Mat}.  Thus a local ring need not be Noetherian.
 An element $x$ in the maximal ideal $\m$ of a regular local ring $R$ is said
to be a {\it regular parameter} if $x \not\in \m^2$.   It then follows that  the residue class ring $R/xR$ is again a regular local 
ring.  
We refer to an extension ring $B$ of an integral domain $A$ as an {\it overring of} $A$ if $B$ is a subring of the quotient field of $A$.  If, in addition,  $A$ and $B$ are local and the inclusion map $A \hookrightarrow B$ is a 
local homomorphism,  we say that $B$ {\it birationally dominates}  $A$.

\section{Essential prime divisors of a  sequence of quadratic transforms} 
 
 


Let $\{R_i\}$ be an infinite  sequence of local quadratic transforms of a regular local ring $R$.  
In this section we consider  the set consisting of the DVRs that are essential prime divisors of infinitely many of 
the $R_i$.  We see later in Proposition~\ref{localizations of S}  that if  the Shannon extension  $S = \bigcup_i R_i$ is not a rank 1 valuation ring,  then this set is precisely the set of localizations of $R$ at the height $1$ prime ideals of $S$. 
A key technical tool in describing these  DVRs, as well as one that we use heavily throughout the rest of the paper, is that of  the transform of an ideal. 
We first review this concept.

Let $R \subseteq S$ be Noetherian UFDs with $S$ an overring of $R$, and 
let    $I$ be  a nonzero ideal of   $R$.  Then the ideal   $I$ can be written uniquely as $I = P_1^{e_1} \cdots P_n^{e_n}J$, where the $P_i$ are principal prime ideals of $R$, the $e_i$ are positive integers and $J$ is an ideal of $R$  not contained in a principal ideal of $R$ \cite[p.~206]{Lip}.  
  For each $i$, set $Q_i = P_i({R \setminus P_i})^{-1}S  \cap S$.  If $S \subseteq R_{P_i}$, then $R_{P_i} = S_{Q_i}$,  and otherwise $Q_i = S$.
  The {\it  transform}\footnote{We are following Lipman's terminology in \cite{Lip}. In the terminology  of Hironaka \cite[Definition 5, p.~213]{Hir}, our notion of the transform is the {weak transform} of $I$. If $I$ is a nonzero principal ideal of the regular local ring $R$, then the weak and strict transforms of $I$ coincide. Thus, in \cite[p.~701]{Gra}, our notion of the transform coincides with the strict transform of $I$, since Granja restricts to principal ideals.} of
   $I$ in $S$ is the ideal $$I^S = Q_1^{e_1} \cdots Q_n^{e_n}(JS)(JS)^{-1}.$$   
Alternatively, $I^S = Q_1^{e_1} \cdots Q_n^{e_n}K$, where $K$ is the unique ideal of $S$ such that both $JS = xK$ for some $x \in S$ and 
$K$ is not contained in a proper principal ideal of $S$.

We recall the following useful result about transforms.  

  \begin{lemma} \label{Lipman} {\em (Lipman \cite[Lemma 1.2 and Proposition 1.5]{Lip})} Let  $R \subseteq S \subseteq T$ 
  be Noetherian UFDs   with  $S$ and $T$  overrings of $R$.  Then
   
  \begin{itemize}
\item[{\em (1)}] $(I^S)^T = I^T$ for all ideals $I$ of $R$.
\item[{\em (2)}] $(IJ)^S = I^SJ^S$ for all ideals $I$ and $J$ of $R$.
\item[{\em (3)}] Suppose that $P$ is a nonzero  principal prime ideal of $R$. Then the following are equivalent.
\begin{itemize}
\item[{\em (i)}] $P^S \ne S$. 
\item[{\em (ii)}] $S \subseteq R_P$. 
\item[{\em (iii)}] $P^S$ is the unique prime ideal $Q$ in $S$ such that $Q \cap R = P$; and $R_P = S_Q$.  
\end{itemize}
\end{itemize}
\end{lemma}

 If  $R$ is a regular local ring and $S$ is a local quadratic transform of $R$, then the order valuation of $R$ can be used to calculate the transform of an ideal of $R$.  How to do this is indicated in Remark~\ref{GR lemma}, but first we recall the construction of the order valuation of $R$.  Let $R$ be a regular local ring with maximal ideal $\m$ and  quotient field $F$. 
For each $0 \ne x \in R $, we define $\ord_R(x) = \min\{i \mid x \in \m^i\}$, and we extend $\ord_R$ to a map from $F$ to ${\mathbb{Z}} \cup \{\infty\}$ by defining $\ord_R(0) = \infty$ and $\ord_R(x/y) = \ord_R(x) - \ord_R(y)$ for all $x,y \in R$ with $y \ne 0$.  From the fact that $R$ is a regular local ring, it follows that $\ord_R$ is a discrete rank one valuation on $F$. The valuation $\ord_R$ is the {\it order valuation} of $R$.  The valuation ring of $\ord_R$ is said to be the {\it order valuation ring} of $R$.  It follows that if $R_1$ is any local quadratic transform of $R$, then $(R_1)_{\m R_1}$ is the order valuation ring of $R$.

 \begin{remark} \label{GR lemma}  
{\em Let $R_1$ be a local quadratic transform of  a regular local ring $R$,
and let $x$ be an element of ${\m}$ such that ${\m}R_1 = xR_1$. 
If $I$ is an ideal of $R$  
and $e = \ord_{R}(I)$, then $I^{R_1} = x^{-e}IR_1$ and ${\m}^eI^{R_1} = IR_1$.  
 In \cite[p.~1349]{GMR2}, this equation is used 
to define the 
 (strict) transform of a height $1$ prime ideal in $R_1$. }
\end{remark}



\begin{definition} \label{epd def}
{\em For an integral  domain $A$, let \begin{center}$\epd(A)= \{A_P \mid P$ is a height $1$ prime ideal of $A\}$. \end{center}
The notation is motivated by the fact that if $A$ is a Noetherian 
integrally closed domain, then  $\epd(A)$ is the set of  essential prime divisors of $A$. 
 With $R$ a regular local ring, let $\{(R_i,\m_i)\}$ be a sequence of local quadratic transforms of $R$ (so that for each $i \geq 0$, ${\m}_i$ is the maximal ideal of $R_i$ and  $R_{i+1}$ is a local quadratic transform of $R_i$), and let $S = \bigcup_{i}R_i$. Define 
$$\epd(S/R) = \big\{W \in \bigcup_{i \geq 0} \epd(R_i) \mid S \subseteq W\big\}.$$ }
\end{definition}

\begin{remark}  \label{the epds}
{\em The set $\epd(S/R)$ consists of the essential prime divisors of $R$ that contain $S$
 along with the order valuation rings of any of the $R_i$ that contain $S$.
This follows, for example, from Lemma~\ref{isom off closed fiber}.

Moreover, $S$ is a rank $1$ valuation domain if and only 
if $\epd (S / R) = \emptyset$ \cite[Proposition 4.18]{Sha},\footnote{In this case, the sequence $\{ R_i \}$  switches  strongly infinitely often.}
and $S$ is a rank $2$ valuation domain if and only if $\epd (S / R)$ consists of a single 
element \cite[Theorem 13]{Gra}.\footnote{In this case, the sequence $\{ R_i \}$ is height 1 directed.}
If $S$ is not a rank $1$ valuation domain, then   Proposition~\ref{localizations of S}  implies that 
$\epd(S/R) = \epd(S)$.}
\end{remark} 

\begin{notation} \label{ideal notation}
{\em In the setting of Definition~\ref{epd def},  there is naturally associated to the sequence 
$\{(R_i, \m_i)\}$ a  sequence  $\{\ff I_i\}$
 of ideals of $R$, 
where each $i \geq 0$, 
\begin{center}
$
{\ff I}_i ~ =  ~R ~ \cap ~    {\m}_0{\m}_1 \cdots {\m}_{i}R_{i+1}. $ 
\end{center}}
\end{notation}
  

\begin{lemma} \label{infinite} 
In the setting of Notation~\ref{ideal notation},
let $P$ be a  height 1 prime ideal of $R$ generated by a regular
 parameter of $R$.  Then  for $k \geq 1$, $R_k \subseteq R_P$ if and only if $P \subseteq {\ff I}_k$.  Thus 
 $S \subseteq R_P$ if and only if $P \subseteq \bigcap_{k > 0 }{\ff I}_k$.

\end{lemma} 

\begin{proof}
Suppose that $R_k \subseteq R_P$. An inductive argument using Lemma~\ref{Lipman}(3) shows that $P^{R_k}$ is a prime ideal of $R_k$, and an inductive argument using 
%
 Remark~\ref{GR lemma} and the transitivity of the transform (Lemma~\ref{Lipman}(1)) shows that  
\begin{eqnarray} \label{GR lemma eqn} {\m}_0 {\m}_1 \cdots {\m}_{k-1}P^{R_{k}} =
 PR_{k}.\end{eqnarray} Therefore,
since $P^{R_k} \subseteq {\m}_kR_{k+1}$, we have  
     $P \subseteq  {\m}_0 {\m}_1 \cdots {\m}_{k-1}{\m}_kR_{k+1}$, from which we conclude that $P \subseteq  {\ff I}_k$.  

Conversely, suppose that $P \subseteq {\ff I}_k$. Along with  (\ref{GR lemma eqn}), this implies  
$${\m}_0 {\m}_1 \cdots {\m}_{k-1}P^{R_{k}}  \subseteq  {\m}_0 {\m}_1 \cdots {\m}_kR_{k+1}.
 $$ Since ${\m}_iR_{i+1}$ is a principal ideal of $R_{i+1}$ for each $i$, we conclude that $P^{R_k} \subseteq {\m}_kR_{k+1}$. Therefore, by Lemma~\ref{Lipman}(3), $R_k \subseteq R_P$.  
\end{proof}

The following classical fact is used without proof in \cite[Lemma 11]{GMR}. Because it will be important in what follows, 
we include a proof.  The proof illustrates  calculations  involved in local quadratic transforms. 
\begin{lemma}  \label{proximate2}
Let $(R, \m)$ be a regular local ring and let $R_1$ be a local quadratic transform of $R$.
We have  $\m R_1 = z R_1$ for some $z \in \m  \setminus \m^2$.   Assume that  $( x_1, \ldots, x_s) R$ is a 
regular prime ideal of $R$ of height $s$ such that $R_{x_i R} \supset R_1$ for each $i \in \{1, \ldots, s\}$.
Then:
\begin{enumerate}
\item  $(z, x_1, \ldots, x_s)R$  is a regular prime ideal of $R$ of height $s+1$,  and 
\item 
$( z, \frac{x_1}{z}, \ldots, \frac{x_s}{z} ) R_1$ is a regular prime ideal of $R_1$ of height $s+1$.
\end{enumerate}
\end{lemma}

\begin{proof}
Since $R_{x_i R} \supset R_1$,  the transform of $x_i R_i$ in $R_1$, which is $\frac{x_i}{z} R_1$, is a regular prime of $R_1$. We  first prove  that $(z, x_1, \ldots, x_s) R$ is a regular prime ideal of $R$ of height $s +  1$.
Assume,  by way of contradiction,  that $z + f \in (x_1, \ldots, x_s) R$ for some $f \in \m^2$.
Then  $\frac{f}{z} \in \m_1$, so $1 + \frac{f}{z}$ is a unit in $R_1$, but $1 + \frac{f}{z} \in (\frac{x_1}{z}, \ldots, \frac{x_s}{z}) R_1$, contradicting the fact that $\frac{x_i}{z} \in \m_1$ for each $i$.  This proves item 1.

We may extend the ideal  $(z, x_1, \ldots, x_s)R$ to  a minimal generating set for $\m$,  say 
 $\m  =  (z, x_1, \ldots, x_s, y_1, \ldots, y_t) R$.
By construction of local quadratic transform, $R_1 / z R_1$ is isomorphic to the localized polynomial ring  
	$$R_1 / z R_1 \cong k \left[ \overline{\frac{x_1}{z}}, \ldots, \overline{\frac{x_s}{z}}, \overline{\frac{y_1}{z}}, \ldots, \overline{\frac{y_t}{z}} \right]_{\p}$$
where $k = R / \m$ and $\p$ is some prime ideal containing $\overline{\frac{x_1}{z}}, \ldots, \overline{\frac{x_s}{z}}$.
Thus $\overline{\frac{x_1}{z}}, \ldots, \overline{\frac{x_s}{z}}$ generate a  regular prime of $R_1 / z R_1$ of
 height $s$.  It follows that 
 $z, \frac{x_1}{z}, \ldots, \frac{x_s}{z}$ generate a regular prime ideal  of $R_1$  of  height $s+1$.
\end{proof}

\begin{proposition}
\label{new order valuations}  \label{finitely many order valuations} 
  In the setting of Definition~\ref{epd def},  
 the set  $\epd(S/R)$ contains at most $\dim R- 1$ of the order valuation rings of the quadratic sequence $\{R_i\}$.
 \end{proposition}  
 
 \begin{proof} 
For each $i$, let $V_i$ be the order valuation ring for $R_i$.   
 Let $d = \dim R $, and suppose by way of contradiction that $V_{i_1}, \ldots, V_{i_d}$, with $i_1 < \cdots < i_d$, contain $S$.  For each $k  \in \{1,\ldots,d\}$, let   $P_{i_k}$ denote the center of $V_{i_k}$ in $R_{i_k + 1}$. Let  
 $j = i_d + 1$, and let $P_{i_k}^{R_j}$ be the transform of $P_{i_k}$ in $R_j$. By Lemma~\ref{Lipman}(3), $P_{i_1}^{R_j},\ldots, P_{i_d}^{R_j}$ are proper ideals of $R_j$. For each $k$, write $P_{i_k}^{R_j} = x_kR_j$. By 
 Lemma~\ref{proximate2},  the elements $x_1,\ldots,x_d$ form part of a regular sequence of parameters of $R_j$. Since $\dim R_j  \le \dim R$ \cite[Theorems 15.5, p.~118]{Mat} and 
 $d = \dim R$, the elements $x_1,\ldots,x_d$  generate the maximal ideal of $R_j$.  Also by Lemma~\ref{infinite}, each $P^{R_j}_{i_k}\subseteq \bigcap_{k > j}{\ff I}_k$, and hence ${\m}_j = \bigcap_{k > j}{\ff I}_k$, which in turn forces ${\m}_j \subseteq {\m}_{j}{\m}_{j+1}R_{j+2}$. Since ${\m}_jR_{j+2}$ is a principal ideal of $R_{j+2}$, we have $R_{j+2} \subseteq {\m}_{j+1}R_{j+2}$, a contradiction.   
 \end{proof}
 
 \begin{remark} \label{dim remark}
 {\em The bound of $\dim R - 1$ in Proposition~\ref{new order valuations} can be refined with the following observation. In the setting of Definition~\ref{epd def}, $\dim R_{i} - \dim R_{i+1}$ is equal to the transcendence degree of the residue field of $R_{i+1}$ over that of  $R_i$ \cite[(1.4.2)]{Abh2}.   Thus, for each $i$, $\dim R_i \geq \dim R_{i+1}$, so that there is positive integer $d\leq \dim R$ with $d = \dim R_i$ for all $i \gg 0$.  The proof of Proposition~\ref{new order valuations} shows that $\epd(S/R)$ contains at most $d-1$ valuation rings. Moreover, since $S$ is an overring of $R$, we have $\dim S \leq d$ \cite[Theorem 15.5, p.~118]{Mat}. 
   }
 \end{remark}

\section{The maximal ideal of a Shannon extension}

As a directed union of local rings, a Shannon extension $S$ is local.
In this section we focus on the maximal ideal $N$ of $S$ and show  that either $N$ is principal or idempotent (Proposition~\ref{maximal ideal}), and that in either case, $N$ is the radical of a principal ideal (Proposition~\ref{flat lemma}). 

\begin{setting} \label{setting 1} {\em We make the following assumptions throughout the rest of the paper. 
\begin{itemize}
\item[{(1)}] $(R, \m)$ is a regular local ring with quotient field $F$ such that $\dim R \ge 2$.

\item[{(2)}] 
$\{(R_i, \m_i)\}$ is an infinite sequence of local quadratic transforms of regular local rings starting from $R_0 = R$.
That is, for each $i > 0$, $R_{i}$ is a local quadratic transform of $R_{i-1}$, so $R_i$ is a regular local ring, $R_{i-1} \subsetneq R_i$, and, by Remark~\ref{dim remark}, $\dim R_{i-1} \ge \dim R_i \ge 2$.

\item[{(3)}] $S = \bigcup_{i = 0}^\infty R_i$ is the Shannon extension of $R$ along $\{R_i\}$ and $N=\bigcup_{i=0}^\infty {\m}_i$ is the maximal ideal of $S$.

\item[{(4)}] For each $i \geq 0$, $\ord_i:F \rightarrow {\mathbb{Z}} \cup \{\infty\}$ represents the order valuation of $R_i$ and $V_i = \{q \in F \mid \ord_i(q) \geq 0\}$ is the corresponding valuation ring.

\end{itemize}
 }
\end{setting}

Lemma~\ref{isom off closed fiber} is well known.
The geometric content of the lemma is that blowing up the maximal ideal $\m$ is an isomorphism outside of the fiber over $\m$.

\begin{lemma}\label{isom off closed fiber}
Assume Setting~\ref{setting 1}.  If $P_1$ is a prime ideal of $R_1$ such that $P:=P_1 \cap R$ is a nonmaximal prime ideal of $R$, then 
 $R_{P} = (R_1)_{P_1}$.
\end{lemma}


\begin{proof} 
There exists a regular parameter $x$ of $R$ such that $R_1$ is a localization of $R[\m/x]$ at a prime ideal 
$Q_1$. If $x \in P$, then ${{\ff m}}R_1 = xR_1 \subseteq  P_1$, and hence ${{\ff m}} = P_1 \cap R$, a contradiction to the assumption that $P=P_1 \cap R$ is a nonmaximal prime ideal of $R$.  
Thus  $x \not\in P$.  It follows that both $R_P$ and $(R_1)_{P_1}$ are localizations of $R[1/x]$.  Since  $(R_1)_{P_1}$
birationally   dominates $R_P$,  we have that  $R_{P} = (R_1)_{P_1}$.
\end{proof}

We see in the next proposition that a Shannon extension has an isolated singularity, in the sense that every non-closed point of $\Spec \: S$ is nonsingular. A stronger version of the proposition is proved in Theorem~\ref{hull}(1), which asserts that the punctured spectrum of $S$ is  a localization  of  
$R_i$ for sufficiently large $i$. 

\begin{proposition}\label{localizations of S}
Assume Setting~\ref{setting 1}.
If $P$ is a nonzero nonmaximal prime ideal of $S$, then $S_P = (R_i)_{P \cap R_i}$ for $i \gg 0$,  and hence $S_P$ is  a regular local ring.
\end{proposition}

\begin{proof}
Let $P$ be a nonmaximal prime ideal of $S$ and denote $P_n = P \cap R_n$, so set-theoretically $P = \bigcup P_n$ and $S_P = \bigcup_{n \ge 0} (R_n)_{P_n}$.
Since $P$ is nonmaximal, $P_n \subsetneq \m_n$ for some fixed large $n$.
An inductive argument with Lemma~\ref{isom off closed fiber} yields that for $m \ge n$, $(R_m)_{P_m} = (R_n)_{P_n}$.
It follows that $S_P = (R_n)_{P_n}$ is a regular local ring.
\end{proof}

It follows from  Theorem~\ref{hull}(1) that the positive integer  $i$ in Proposition~\ref{localizations of S} can be chosen independently of $P$.

In light of Proposition~\ref{localizations of S},  the ideals of the Shannon extension $S$ that are primary for  the maximal ideal play an important role in our treatment of the structure of $S$.  We characterize in the next lemma and proposition when the maximal ideal of $S$ is principal.

\begin{lemma}  \label{principal case}
Assuming Setting~\ref{setting 1}, the following are equivalent for $x \in S$.

\begin{itemize}

\item[{\em (1)}] $N = xS$.

\item[{\em (2)}]  $P:=\bigcap_{i>0}N^i$ is a prime ideal, $S/P$ is a DVR with maximal ideal
the image of  $xS$  and $P = PS_P$.  

\item[{\em (3)}] For every valuation ring $V$ that birationally dominates $S$,  ${\m}_iV = xV$ for all $i \gg 0$. 
 
\item[{\em (4)}]  The element $x$ is a regular parameter in $R_i$ for all  $i \gg 0$. 

\end{itemize}
\end{lemma}


\begin{proof}  The equivalence of (1) and (2) is a standard argument involving only the fact that $S$ is a  local domain; see \cite[Exercise 1.5, p.~7]{Kap}

(1) $\Rightarrow$ (3)
Let $V$ be a valuation ring that birationally dominates $S$.  
  If $i $ is such that $x \in {\m}_i$, then $xV \subseteq {{\m}}_iV \subseteq NV = xV$, and hence $xV = {{\m}_j}V$ for all $j \geq i$. 

(3) $\Rightarrow$ (4)  Let $i$ be such that both $x \in {\m}_i$ and  ${\m}_jV = xV$ for all $j \geq i$. Then since ${\m}_j^2V \subsetneq {\m}_jV$, it follows that $x \in {\m}_j \smallsetminus {\m}_j^2$. Hence $x$ is a regular parameter in $R_j$. 

(4) $\Rightarrow$ (1)  Let $i$ be such that $x$ is a regular parameter for all $j \geq i$. 
Let $j \geq i$.  Then since $x$ is a regular parameter in $R_{j+1}$ and $xR_{j+1}$ is contained in the height 1 prime ideal ${\m}_jR_{j+1}$ of $R_{j+1}$, it follows that $xR_{j+1} = {\m}R_{j+1}$.  Since this holds for all $j \geq i$, we conclude that $N = \bigcup_{j \geq i} {\m}_jR_{j+1} = \bigcup_{j \geq i} xR_{j+1} = xS$. 
%
%
\end{proof}

Following \cite{HK}, we say there is {\it no change of direction} for the quadratic sequence $R_0 \subseteq R_1 \subseteq \cdots \subseteq R_n$ if ${\m}_0 \not \subseteq {\m}_n^2$; otherwise, if ${\m}_0 \subseteq {\m}_n^2$, there is a {\it change of direction} between $R_0$ and $R_n$.  We say that the quadratic sequence $\{R_i\}$ {\it changes direction infinitely many times} if there exist infinitely many positive integers $i$ such that there is a change in direction between $R_i$ and $R_{n_i}$ for some $n_i > i$.

\begin{proposition} \label{maximal ideal} Assuming Setting~\ref{setting 1}, the following statements are equivalent. 
\begin{itemize}

\item[{\em (1)}]   $N$ is not  a principal  ideal of $S$.

\item[{\em (2)}] $N = N^2$.

\item[{\em (3)}]   $\{R_i\}$ changes directions infinitely many times.

\item[{\em (4)}] For every nonzero  element $x$ of $N$ and every $n>0$, $\ord_i(x) >n$ for $i \gg 0$.  
\end{itemize}
\end{proposition} 

\begin{proof}
(1) $\Rightarrow$ (2)  If $N$ is not a principal ideal of $S$, then by Lemma~\ref{principal case}, for each $x \in N$, there exists $i \geq 0$ such that $x \in R_i$ but $x$ is not a regular parameter in $R_i$. 
Hence $x \in {\m}_i^2 \subseteq N^2$, which shows that $N = N^2$. 

(2) $\Rightarrow$ (3) 
For each $i$, let $x_i$ be a regular parameter for $R_i$ such that $x_i R_{i+1} = {\m}_i R_{i+1}$.  
If the maximal ideal $N$ of $S$ is idempotent, then Lemma~\ref{principal case} implies that for each $i$, there exists $n_i > i$ such that $x_i$ is not a regular parameter in $R_{n_i}$.  Thus ${\m}_iR_{i+1} = x_iR_{i+1}  \subseteq  {\m}_{n_i}^2$, and hence there is a change of direction between $R_i$ and $R_{n_i}$. Since this holds for each choice of $i$, we conclude that $\{R_i\}$ changes directions infinitely many times.

(3) $\Rightarrow$ (4)  Let $x \in N$. Since $\{R_i\}$ 
 changes directions infinitely many times, it follows that for each $i$ there exists $j>i$ such that ${\m}_i \subseteq {\m}_j^2$ and hence $\ord_i(x) < \ord_j(x)$.  Hence for each $n >0$ there exists $i$ such that $\ord_j(x) >n$ for all $j \gg i$.  
 
 (4) $\Rightarrow$ (1)  Suppose that $N = xS$ for some $x \in N$. Then by  Lemma~\ref{principal case}, $x$ is a regular parameter in $R_j$ for $j \gg 0$.    This implies that $\ord_j(x) = 1$ for $j \gg 0$, so that $\{i\mid\ord_j(x) > 1\} $ is a finite set.
\end{proof}

\begin{corollary} \label{d-1} Assume Setting~\ref{setting 1}. If $N$  is not principal, then every valuation ring between $S$ and its field of fractions has rank at most $\dim R-1$, and hence $\dim S < \dim R$.    
\end{corollary}


\begin{proof}
Let $U$ be a valuation ring { between} $S$ and its field of fractions $F$. By Proposition~\ref{maximal ideal},  $NU$ is an idempotent ideal of $U$. Since $U$ is a valuation ring, this implies $NU$ is a prime ideal of $U$  \cite[Theorem~17.1, p.~187]{Gil}.  A rank $d$ valuation ring between a $d$-dimensional Noetherian ring and its field of fractions is discrete \cite[Theorem 1]{Abh},  and hence has no nonzero idempotent prime ideals. Thus the rank of $U$ is at most $\dim R  -1$.  Since the rank of every valuation ring between $S$ and $F$ is at most $\dim R -1$, it follows that $\dim S < \dim R$ \cite[(11.9), p.~37]{Nag}.  
\end{proof}

\begin{remark} {\em 
 In contrast to Corollary~\ref{d-1}, if the maximal ideal of $S$ is principal it need not be true that $\dim S < \dim R$. 
 If $\dim R = 2$, then every rank $2$ valuation ring that birationally dominates $R$ is a Shannon extension $S$ with $\dim S = \dim R = 2$; see Corollary~\ref{Abhyankar cor}. There also exist examples with $\dim S = \dim R$  in which $S$ is not a valuation ring: the Shannon extension $S$ in 
 Example~\ref{Shannon example 4.7} 
is not a valuation ring and $\dim S= \dim R = 3$.     }
\end{remark}

\begin{proposition} \label{flat lemma} Assuming Setting~\ref{setting 1},   
 there exists a regular parameter $x$ in one of the $R_i$'s such that $xR_{i+1} = {\ff m_i}R_{i+1}$ and 
 $xS$ is an $N$-primary ideal of $S$.  
\end{proposition} 

\begin{proof}
By Proposition~\ref{new order valuations}, there exists $i \geq 0$ such that no order valuation ring $V_j$, $j \geq i$,  
  is in $\epd(S/R_i)$, so $\epd(S/R) \subseteq \epd(R_i)$.
Let $x \in \m_i$ be such that $x R_{i+1} = \m_i R_{i+1}$.
Note that $\m_i \subseteq x S$ and $x S \cap R_i = \m_i$.
Assume by way of contradiction that a non-maximal prime ideal $\q$ of $S$ contains $x$.
Then $S_{\q}$ is Noetherian by Proposition~\ref{localizations of S}, so there exists a height $1$ prime ideal $\p$ of $S$ such that $x \in \p$.
Then
$$\m_i = x S \cap R_i \subseteq \p \cap R_i \subseteq N \cap R_i = \m_i,$$
so $\p \cap R_i = \m_i$.
However, $S_{\p} \in \epd(S/R) \subseteq \epd(R_i)$.
Thus $S_{\p} = (R_i)_{\p \cap R_i} = R_i$, contradicting the assumption that $\dim R_i \ge 2$.
We conclude that $\sqrt{x S} = N$.

 \end{proof}

\begin{corollary}  \label{DVR case} Assuming Setting~\ref{setting 1},
the following are equivalent for the Shannon extension $S$ of $R$. 

\begin{itemize}
\item[{\em (1)}] 
 $S$ is dominated by a DVR.

\item[{\em (2)}] $S$ is a DVR.

\item[{\em (3)}] 
$S$ is a Noetherian ring.

 \end{itemize}
 \end{corollary}
 
 \begin{proof} 
(1) $\Rightarrow$ (2)\footnote{The argument given here is related to classical results of Abhyankar in \cite{Abh} and
Zariski in  \cite[pp.~27-28]{Zar54}.}    Suppose $V$ is a DVR that dominates $S$. Replacing $V$ by $V \cap F$, 
 we may assume that $V$ birationally dominates $S$. 
We claim that $S = V$.  
Let $ f \in  V$.  Since $R$ is a UFD, we can write $f = a/b$, where $a,b \in R$   are  relatively prime. 
 Let $v$ denote the valuation associated to $V$ with value group
the integers.  We have $v(f) = v(a) - v(b) \ge 0$,  and  $v(b) = 0$ if and only if  $f \in R$.   Assume that
$v(b) = n > 0$. Let $x \in {\m} $ be
such that ${\m}V = xV$.  Then $x$ is part of a regular system of parameters for $R$ and
 $xR_1 = {\m}R_1$.   Hence there exist $c,d \in R_1$ such that $a = xc$ and $b = xd$. 
 If follows that $f = c/d$ and $v(d) < n$.
Writing $c/d$ in lowest terms in $R_1$ will not increase the $v$-value of the denominator.  Hence
repeating this process at most $n$ times gives $f \in R_i$, with $i \le n$. Thus $S = V$.   

(2) $\Rightarrow$ (3) This is clear.

(3) $\Rightarrow$ (1)  By Proposition~\ref{flat lemma}, $N$ is the radical of a principal ideal of $S$, and hence since $S$ is Noetherian, $\dim S = 1$. Moreover, since $N$ is finitely generated, $N$ is not idempotent, and hence by Proposition~\ref{maximal ideal}, $N$ is a principal ideal. Thus $S$ is a DVR.
\end{proof}

\begin{remark} {\em Another condition that characterizes when a Shannon extension is a DVR is given in Corollary~\ref{another DVR}.} 
\end{remark}

\section{The Noetherian hull  of a Shannon extension}

In this section we continue to assume Setting~\ref{setting 1}, and we show that there is a smallest Noetherian overring $T$ that properly contains the Shannon extension $S$ and this ring is a regular ring that is a localization of $R_i$ for sufficiently large $i$. In the next section $T$ is used to decompose $S$ as an intersection of a regular ring and a valuation ring.

\begin{theorem} \label{hull} Assuming Setting~\ref{setting 1}, let $T$ be the intersection of all the DVRs with quotient field $F$  that properly contain $S$, 
 where an empty intersection equals $F$.
  Then the following statements hold for $T$.
\begin{itemize}

\item[{\em (1)}]  $T = S[1/x]$ for any $x \in N$ such that $xS$ is $N$-primary.
Furthermore, $T$ is a localization of $R_i$ for $i\gg 0$. In particular, $T$ is a UFD.

\item[{\em (2)}] The ring $T$ is the intersection of all the $R_P$, $P$ a height 1 prime ideal of $R$, that contain $S$, along with the  at most $\dim R - 1$  order valuation rings $V_i$ that contain $S$.  

\item[{\em (3)}] 
The ring $T$ is a Noetherian regular ring that is the unique minimal   proper  Noetherian overring of $S$ 
in $F$.\footnote{If $\dim S = 1$,   then  $T$ is the quotient field of $S$ by Proposition~\ref{localizations of S} and statement (2). 
 We regard a field to 
be a zero-dimensional regular local ring.}

 \end{itemize} 
\end{theorem} 

\begin{proof}
If $S$ is a DVR, then   $T$ is the quotient field of $S$ and the assertions  (1), (2) and (3) hold, 
so we assume that $S$ is not a DVR.  

(1) Let $x \in S$ such that $xS$ is $N$-primary. 
By Proposition~\ref{new order valuations}, there exists $i \geq 0$ such that no order valuation ring $V_j$, $j \geq i$,  
  is in $\epd(S/R_i)$.  Thus no order valuation ring of the sequence $\{R_j\}_{j \geq i}$ is in $\epd(S/R_i)$, and hence no order valuation ring of this sequence contains $S$. 
   Therefore,   
   by replacing $R$ with $R_i$ we may assume that $\epd(S/R)$ contains none of the $V_i$. 

By Proposition~\ref{flat lemma} there exists a regular parameter $x_i$ 
in one of the $R_i$'s such that $x_iR_{i+1} = \m_iR_{i+1}$ and $x_iS$ is $N$-primary. We may assume without loss of generality that 
 $i = 0$ and that $x = x_i$; in particular,  $N = \sqrt{xS}$ and  $xR_1 = {\m}R_1$. We claim that $S[1/x]$ is a flat extension of $R$. To prove this, it is enough to show that for each prime ideal $P$ of $S$ that survives in $S[1/x]$, $S_P = R_{P \cap R}$ \cite[Theorem 2]{Ric}. Let $P$ be such a prime ideal. Then since $x \not \in P$ and $x \in {\m}_i$ for all $i \geq 0$, it must be that for each $i$,  $P\cap R_i$ is a nonmaximal ideal prime ideal of $R_i$.  Therefore, by Lemma~\ref{isom off closed fiber}, for each $i \geq 0$, we have $R_{P \cap R} = (R_{i})_{P \cap R_i}$, and hence $R_{P \cap R} = \bigcup_{i\geq 0}(R_{i})_{P \cap R_i} =  S_P$, which proves the claim. 

Next, since  $S[1/x]$ is a flat extension of the UFD $R$, then $S[1/x]$ is a localization of $R$ at a multiplicatively closed set \cite[Theorem 2.5]{HR}.  
To complete the proof of (1), we claim that $T = S[1/x]$. Since $S[1/x]$, as a localization of $R$, is an integrally closed Noetherian domain, it is an intersection of DVRs and hence $T \subseteq S[1/x]$.  It remains to show that every DVR that contains $S$ contains $S[1/x]$. Let $V$ be a DVR that contains $S$. If $V$ dominates $S$, then by Corollary~\ref{DVR case}, $S$ is a DVR, contrary to assumption. Thus $V$ does not dominate $S$ and since $xS$ is $N$-primary, it follows that $S[1/x] \subseteq V$, which proves that $S[1/x] = T$.

(2)  By (1), there is $i \geq 0$ such that $T$ is a localization of $R_i$, and thus 
$T$ is an intersection of the $(R_i)_P$ that contain $S$, where $P$ is a height 1 prime ideal of $R_i$. It follows from 
 Lemma~\ref{isom off closed fiber} that each $(R_i)_P$ is a localization of $R$ at a height 1 prime ideal of $R$ or $(R_i)_P$ is an order valuation of some $R_j$, $j < i$.  By Proposition~\ref{new order valuations}, there are at most $(\dim R) -1$ such order valuation rings. 

(3) By (1), $T$ is a localization of a regular local ring and hence is a 
  Noetherian regular ring. 
  Suppose that $A$ is a Noetherian overring of $S$.  Let $M$ be a maximal ideal of $A$. If $M \cap S = N$, then since there is a DVR that dominates the Noetherian ring $A_M$, this DVR dominates also $S$, which by 
Corollary~\ref{DVR case} implies that 
$S$ is a DVR, contrary to our assumption at the beginning of the proof.  Thus $T = S[1/x] \subseteq S_{M \cap S} \subseteq A_M$. Since this is true for every maximal ideal $M$ of $A$, it follows that $T \subseteq A$, which verifies (3). 
\end{proof} 

\begin{definition}
{\em In light of Theorem~\ref{hull}(3), we define the Noetherian regular  UFD   $T$ of  Theorem~\ref{hull}  to be the  {\it Noetherian hull} of the Shannon extension $S$. } 
 \end{definition}



\begin{remark} \label{last remark}
{\em Assume the notation of Setting~\ref{setting 1} and assume $\dim S > 1$. 
Proposition~\ref{new order valuations} implies that for $n \gg 0$, $S$ does not contain the order valuation ring $V_n$.
The proof of Theorem~\ref{hull}(1) shows that for every height $1$ prime ideal $P$ of $S$ we have $S_P = (R_n)_{p R_n} = T_{pT}$ for some prime element $p$ of $R_n$.
Notice, however, that for $i > n$, the ideal $pR_i$ is not a prime ideal.}
\end{remark}

\begin{proposition}\label{4.4} 

With notation as in Theorem~\ref{hull}, fix $n \gg 0$ such that $T$ is a localization of $R_n$.
Let $a \in R_n$ be nonzero, and for $i \ge n$, consider the transform $(a R_n)^{R_i}$.
Then $(a R_n)^{R_i} = R_i$ for $i \gg 0$ if and only if $a \in T^{\times}$.\end{proposition}

\begin{proof}
Since $\m_i T = T$ for all $i \ge n$, we have $a T = (a R_n)^{R_i} T$.
If $(a R_n)^{R_i} = R_i$ for some $i \ge n$, then $a T = T$.

To see the converse, assume that $(a R_n)^{R_i} \subsetneq R_i$ for all $i \ge n$.
By construction of transform, the height $1$ primes of $(a R_n)^{R_{i+1}}$ lie over height $1$ primes of $(a R_n)^{R_i}$.
This yields an ascending sequence $\{ \p_{i} \}_{i \ge n}$, where $\p_i$ is a height $1$ prime of $R_i$.
We have $(a R_n)^{R_i} \subset \p_{i}$ and $\p_{i+1} \cap R_i = \p_i$ for all $i \ge n$.
It follows that $P = \bigcup_{i \ge n} \p_i$ is a height $1$ prime in the directed union $S$, and $a \in P$, so $a \notin T^{\times}$.
\end{proof}


\section{The boundary valuation of a Shannon extension}

Let ${\ff X}$ denote the set of all valuation overrings
of $R$. The Zariski topology on ${\ff X}$ has as a basis of open sets the sets of the form $\{V \in {\ff X} ~| ~ E  \subseteq V\}$, 
where $E$ ranges over the finite subsets of the quotient field $F$   of $R$. For our purposes we need a finer topology: The {\it patch topology} on $\mathfrak{X}$ has a basis of open sets of the form
 $\{V \in {\ff{X}}~ | ~ G \subseteq V$ and $H \subseteq {\ff M}_V\}$, where $G$ and $H$ range over all finite subsets of $F$ \cite{Hoc}.    

\begin{definition}  {\em Assume Setting~\ref{setting 1}. 
A valuation overring $V$ of $R$ is a {\it boundary valuation ring} of $S$ if in the patch topology $V$ is a limit point of the order valuation rings $V_i$.}
\end{definition}

 The terminology is explained by the fact that the subspace $\{V_i\mid i\geq 0\}$ is discrete in the patch topology and hence a valuation ring $V \in {\ff X}$ is a boundary valuation ring of $S$ if and only if $V$ is a boundary point in ${\ff{X}}$ of the set $\{V_i\mid i \geq 0\}$  with respect to the patch topology. 
   Equivalently, $V \in {\ff X}$ is a boundary valuation ring of $S$ if and only if for each pair of finite subsets $G  \subseteq V$ and $H \subseteq {\ff M}_V$, there exist infinitely many $i$ such that $G \subseteq V_i$ and $H \subseteq {\ff M}_{V_i}$.  
 
In Corollary~\ref{equation V cor}, we show that  $S$ has a unique boundary valuation ring, and we use this valuation ring in Theorem~\ref{flat} to give an intersection decomposition of $S$ in terms of its boundary valuation ring and Noetherian hull.

\begin{lemma}\label{compare}
Assuming Setting~\ref{setting 1}, let $q \in F$ be nonzero.
Then either $\ord_{n} (q) > 0$ for $n \gg 0$, $\ord_{n} (q) = 0$ for $n \gg 0$, or $\ord_{n} (q) < 0$ for $n \gg 0$.
\end{lemma}

\begin{proof}
If $q \in R_n$ or $q^{-1} \in R_n$ for some $n \ge 0$, then  Lemma~\ref{compare}   is clear.
Assume that $q \notin R_n$ and $q^{-1} \notin R_n$ for all $n \ge 0$.



We may write $q = a_0 / b_0$, where $a_0, b_0 \in R$ are relatively prime.
Since $q \notin R$ and $q^{-1} \notin R$, we must have that $a_0, b_0 \in \m$.
Let $Q_0 = (a_0, b_0) R$, so $Q_0$ is an ideal of $R$ of height $2$.
Let $\{ Q_i \}_{i=0}^{\infty}$ be the sequence of transforms of $Q_0$ in the sequence $\{ R_i \}$; i.e., for each $i \ge 0$, $Q_{i+1}$ is the transform of $Q_i$ in $R_{i+1}$.
For each $i$, 
let $x_i \in \m_i$ such that $x_{i} R_{i+1} = \m_i R_{i+1}$.
Then by Remark~\ref{GR lemma}, $Q_{i} R_{i+1} = x_{i}^{e_i} Q_{i+1}$, where $e_i = \ord_{i} (Q_i)$. Let $\{a_i\}$ and $\{b_i\}$ be sequences of elements of $S$ defined inductively for $i \ge 0$ by $a_{i+1} = x_i^{-e_i} a_i$ and $b_{i+1} = x_{i}^{- e_i} b_i$. It follows that $Q_{i} = (a_i, b_i) R_i$ and $q = a_i / b_i$ for all $i \ge 0$.
If $e_{i} = 0$ for any $i \ge 0$, then one of either $a_{i}$ or $b_{i}$ is a unit in $R_{i}$, so $q \in R_i$ or $q^{-1} \in R_i$.
Thus we may assume that $e_i > 0$ for all $i \ge 0$.

By \cite[Lemma 3.6 and Remark 3.7]{HKT}, $\ord_i (Q_i) \ge \ord_{i+1} (Q_{i+1})$ for all $i \ge 0$, so the sequence $\{ e_i \}$ is a nonincreasing sequence of non-negative integers. Thus $\{e_i\}$ stabilizes to some value $e > 0$, say $\ord_i (Q_i) = e$ for all $i \ge N$.

Notice that if $\ord_{i} (a_i) = \ord_{i} (Q_i)$, then $a_{i+1} R_{i+1}$ is the transform of the principal ideal $a_i R_{i}$ in $R_{i+1}$, so again by \cite[Lemma 3.6 and Remark 3.7]{HKT}, $\ord_{i+1} (a_{i+1}) \le \ord_{i} (a_i)$.
Therefore if $\ord_{j} (a_j) = e$ for any $j$, then $\ord_{i} (a_i) = e \le \ord_{i} (b_i)$ for all $i \ge j$.
Similarly if $\ord_{j} (b_j) = e$ for any $j$, then $\ord_{i} (b_i) = e \le \ord_{i} (a_i)$ for all $i \ge j$.
Since either $\ord_{N} (a_N) = e$ or $\ord_{N} (b_N) = e$,   the lemma follows.
\end{proof}

\begin{corollary} \label{equation V cor}
Assume Setting~\ref{setting 1}.
The Shannon extension $S$ has a unique boundary valuation ring $V$,  and 
	\begin{equation}\label{equation V}
	V ~   =   ~    \bigcup_{n \ge 0} \bigcap_{i \ge n} V_i = \{ q \in F ~ | ~  \ord_i (q) \ge 0 \text{ for } i \gg 0 \}.\end{equation}
\end{corollary}

\begin{proof} Let $V' =  \{ q \in F ~ | ~  \ord_i (q) \ge 0 \text{ for } i \gg 0 \}$. 
As a directed union of rings, $V'$ is a ring, and in view of Lemma~\ref{compare}, $V'$ is in fact a valuation ring.
Let $V$ be a boundary valuation ring of $S$.
Then for $q \in V$, $q \in V_i$ for infinitely many $i$, so by Lemma~\ref{compare}, $q \in V_i$ for all $i \gg 0$.
Thus $V \subseteq V'$.
Furthermore, for $q \in {\ff M}_{V}$, $q \in {\ff M}_{V_i}$ for infinitely many $i$, so similarly, $q \in {\ff M}_{V_i}$ for all $i \gg 0$.
Thus ${\ff M}_{V} = {\ff M}_{V'} \cap V$, so $V = V'$.
\end{proof}

\begin{theorem} \label{flat}  Assume Setting~\ref{setting 1}. Then $S = V \cap T$, where $V$ is the unique boundary valuation ring of $S$  and $T$ is the Noetherian hull of $S$. 
\end{theorem} 

\begin{proof} 

  First we observe that $S \subseteq V \cap T$. For this it is clearly enough to verify that $S \subseteq V$.  Let $s \in S$. 
  Then there exists $i$ such that $s \in R_i$. Since $R_i$, hence $s$, is contained in every $V_j$, $j \geq i$,  we have by Corollary~\ref{equation V cor}
  that $s \in V$.  Thus 
    $S \subseteq   V \cap T$. 
It remains to prove that $V \cap T \subseteq S$.  If $S$ is a DVR, then since $S \subseteq V \cap T$ and the only proper overring of $S$ is the quotient field of $S$, we have $S = V \cap T$. (Note that $V \ne F$, since by Corollary~\ref{equation V cor}, $V$ dominates $R$.) 
Thus we assume for the rest of the proof  that $S$ is not a DVR.

By Proposition~\ref{new order valuations}, there exists $k > 0$ such that none of the $\{V_i\mid i \geq k\}$ contain $S$.  Thus by replacing $R$ with $R_k$ we may assume without loss of generality that none of the $V_i$ contains $S$.  By Theorems~\ref{flat lemma} and~\ref{hull}(1) there exist $i > 0$ and a regular parameter $x$ in $R_i$ such that ${\m}_{i}R_{i+1} = xR_{i+1}$, $xS$ is $N$-primary and 
$T = S[1/x]$. By replacing $R$ with $R_i$ we may assume without loss of generality that $i = 0$, so that ${\ff m}R_1 = xR_1$.  


Let $q \in V \cap T$, and write $q = s/x^e$ for some $s \in S$ and $e>0$.  By Corollary~\ref{equation V cor} there exists $n>0$ such that $q \in V_i$ for all $i > n$.  
Since $s \in S$ and none of the order valuation rings $V_i$ contain $S$, we may  
choose $k >n$ such that 
\begin{center} $s \in R_k$ \ \ and \ \  $R_k \not \subseteq V_i \ \:  \forall  i =0,1,2,\ldots,n$. 
\end{center}
We claim $q \in R_{k}$.  Since $R_{k}$ is a Krull domain and hence an intersection of its localizations at height $1$ prime ideals,  it suffices to show   that  $q \in (R_{k})_Q$ for each height  $1$ prime ideal $Q$ of $R_{k}$.  
Let $Q$ be a height 1 prime ideal of $R_{k}$.
If $x \not \in Q$, then clearly $q =s/x^e \in (R_{k})_Q$.  
Suppose $x \in Q$. 
Since $\m R_1 = x R_1 \subseteq x R_k \subseteq Q$, it follows that $Q \cap R = \m$.
Thus $(R_k)_{Q} \notin \epd (R)$, so $(R_k)_{Q} = V_i$ for some $i \ge 0$.
Since $R_k \subseteq V_i$, it follows that $i > n$, so $q \in V_i$.
\end{proof}

\begin{corollary} Assume Setting~\ref{setting 1}. If $\dim S = 1$, then $S$ is the boundary valuation ring $V$. 
\end{corollary} 

\begin{proof} Let $0 \ne x \in N$. By Theorem~\ref{hull}, $T = S[1/x]$, so, since $\dim S = 1$, we have $T = S[1/x] =  F$. By Theorem~\ref{flat}, $S$ is the boundary valuation ring $V$.  
\end{proof}

\begin{corollary} \label{chain} Assume  Setting~\ref{setting 1}.  Then  the principal $N$-primary ideals of $S$ are
 linearly ordered with respect to inclusion.
\end{corollary}

\begin{proof}  Let $y,z \in S$ be such that $yS$ and $zS$ are $N$-primary ideals. By   Theorem~\ref{flat},  we have   
 $S = V \cap T$,   where $V$ is the boundary valuation ring of $S$ and $T $ is the Noetherian hull of $S$.     By Theorem~\ref{hull}(1),   $T = S[1/y] = S[1/z]$.  It follows that 
 $yS = yV \cap T$ and $zS = zV \cap T$.  Since   $V$ is  a valuation ring, the ideals $yV$ and $zV$ 
 are comparable,  and thus so are  $yS$ and $zS$.  
\end{proof}

\begin{remark} {\em Using different techniques from the current paper, we prove in \cite{HLOST} that the boundary valuation ring of a Shannon extension always has rank at most  2, and we give constraints on its value group. This is done through an analysis of asymptotic properties of the sequence of local quadratic transforms that defines the Shannon extension. In particular, we recover a result due to Granja \cite[Proposition 7]{Gra}: 
if $S$ is a valuation ring, then $S$ is the boundary valuation ring of $S$ and, in this case, $S$ has at most rank 2. }
\end{remark}




\section{The complete integral closure of a Shannon extension} 

An element $\theta$ in the field of fractions of an integral domain $A$ 
 is  {\it almost integral} over $A$ if $A [\theta]$ is a fractional ideal of $A$.
The integral  domain $A$ is \emph{completely integrally closed} if it contains all almost integral elements in its field of fractions. The \emph{complete integral closure} of a domain is the ring of almost integral elements in its field of fractions. In general, the complete integral closure of a domain may fail to be completely integrally closed.

In this section we describe the complete integral closure of a Shannon extension. To do so, we distinguish between two classes of Shannon extensions, those that are archimedean and those that are not.  Recall that 
an integral domain $A$ is {\it archimedean} if for each nonunit $a \in A$, we have $\bigcap_{n>0}a^nA = 0$. 
A  Shannon extension $S$ is archimedean if and only if $\bigcap_{n>0}x^nS = 0$, where $x$ is as in Setting~\ref{setting 2}.

 To simplify hypotheses, we fix some notation for this section.

\begin{setting} \label{setting 2}    {\em In addition to Setting~\ref{setting 1},   assume:
\begin{itemize}
\item[(1)] 
$x \in S$  is such that $xS$ is $N$-primary (see Proposition~\ref{flat lemma});
\item[(2)]
$T$ is the Noetherian hull of $S$; 
\item[(3)]
$T$ is a localization of $R$ (see Theorem~\ref{hull});
\item[(4)] $S^*$ denotes the complete integral closure of $S$;

\item[(5)] $W$ is the  rank one valuation overring of the boundary valuation ring $V$.  
\end{itemize}
}
\end{setting}

In Setting~\ref{setting 2}, in the special case in which $S$ is a rank one valuation ring, then $S = S^* =  V = W$ and $ S^*  = S = (N:_FN)$.  


\begin{theorem}\label{cic arch}  Assume  notation as in Setting~\ref{setting 1} and \ref{setting 2}.   
Also assume that $S$ is archimedean  and not a DVR.  Let   $V$ be the  boundary valuation ring of $S$.
Then $NV$ is the height $1$ prime of $V$ and hence $W = V_{NV}$. 
Furthermore,
\begin{enumerate}
\item
 $S^* =  (N :_F N) = W \cap T  $, and 
 \item 
 $W$ is the unique rank 1  valuation overring of $S$ with this property. 
 \end{enumerate}
\end{theorem}

\begin{proof}  
 If $N = xS$ is a principal ideal, then  Lemma~\ref{principal case} implies that 
 $\bigcap_{n>0}N^n = \bigcap_{n > 0}x^nS$ is a nonzero ideal, a contradiction to the assumption  that $S$ is
  archimedean.   Thus $N$ is not a principal ideal of $S$. 
 By Proposition~\ref{maximal ideal}, $N$ is an idempotent ideal of $S$, thus $N V$ is an idempotent ideal of $V$.
It follows that $N V$ is a prime ideal \cite[Theorem~17.1, p. 187]{Gil}.
For an $N$-primary element $x$, since $S$ is archimedean, it must be that $\bigcap_{n=0}^{\infty} x^n S = (0)$.
From Theorem~\ref{flat}, we have $S = V \cap T$, and since by Corollary~\ref{DVR case}, 
 $V$ is not a DVR, we have from Theorem~\ref{hull}(1) that $NT = T$, and in particular $x^n T = T$.
It follows that $\bigcap_{n=0}^{\infty} x^n V = (0)$.
Thus $N V$ is the height $1$ prime of $V$, and hence $W = V_{NV}$.  

We show that $(N :_F N) = W \cap T$.
We have,
	$$N V \cap T = N V \cap (V \cap T) = N V \cap S = N$$
so we have the equality
	$$(N :_F N) = ((N V \cap T) :_F N).$$
By properties of colon ideals, it follows that
	$$((N V \cap T) :_F N) = (N V :_F N) \cap (T :_F N).$$
The fact that $NT = T$ implies that $(T :_F N) = T$.
Since $NV$ is idempotent, $W=(NV:_F N)$  \cite[Lemma~4.4, p. 69]{FS}.
We conclude that
	$$(N V :_{F} N) \cap (T :_{F} N) = W \cap T.$$
    Therefore we have established that $(N :_F N) = W \cap T$.




Next we observe that the complete integral closure of $S$ is $(N:_FN)$.  Indeed, since $(N:_FN)$ is the intersection of the completely integrally closed rings $W$ and $T$, the ring $(N:_FN)$ is completely integrally closed. Since also $(N:_FN)$ is a fractional ideal of $S$, $(N:_F N)$ is almost integral over $S$, proving (1).

Since  $ S^* = W \cap T$  and  $\dim W = 1$,  to prove (2),  
by \cite[Corollary 1.4]{HO}, it suffices to show that $W$ cannot be omitted from this representation of $S^*$, or equivalently, that $(N :_F N) \subsetneq T$.
To see this, let $x \in S$ be any $N$-primary element.
Then $\frac{1}{x^2} \in T = S [\frac{1}{x}]$, but $\frac{1}{x^2} x = \frac{1}{x} \notin N$, so $\frac{1}{x^2} \in T \setminus (N :_F N)$.
Thus $W$ cannot be omitted from $S^* = W \cap T$, and hence $W$ is the unique valuation  overring of $S$ of Krull dimension one such that   $ S^* = W \cap T$.
\end{proof}

\begin{corollary} \label{cic N}  With notation as in Theorem~\ref{cic arch},  
$N$ is the center of $W$ on $S^*$.
In particular, $N$ is a prime ideal of $S^*$.
\end{corollary} 

\begin{proof}  
By Theorem~\ref{cic arch}, $S^*  = W \cap T$, and by Theorem~\ref{flat}, $S = V \cap T$.
Using the fact that $N V = N V_{NV}$ set-theoretically \cite[11.2, p. 35]{Nag}, it follows that $N$ is the center of $W$ on $S^{*}$.
\end{proof}




\begin{corollary}\label{same units}
With notation as in Theorem~\ref{cic arch}, the units of $S^*$ are equal to the units of $S$.
\end{corollary}

\begin{proof}
Since the maximal ideal $N$ of $S$ is a proper ideal of $S^{*}$, if $u \in S$ is not a unit of $S$, then $u$ is  also not a unit of $S^{*}$.
Thus it suffices to show that if $u, u^{-1} \in S^{*}$, either $u \in S$ or $u^{-1} \in S$.
Since $S^{*} = W \cap T$ by Theorem~\ref{cic arch}, it follows that $u, u^{-1} \in T$, and at least one of $u, u^{-1}$ is in $V$, say $u \in V$.
But $S = V \cap T$ by Theorem~\ref{flat}, so $u \in S$, completing the proof.
\end{proof}

\begin{corollary}\label{cic subrings}
Assume notation as in Theorem~\ref{cic arch}.
The subrings $A$ of $S^*$ that contain $S$ are in a
one-to-one inclusion preserving correspondence with the subrings of $S^*/N$ that contain the field $k = S/N$.
In particular, $S^*$ is a finitely generated $S$-algebra if and only if $S^* / N$ is a finitely generated $k$-algebra.
\end{corollary}

\begin{corollary} \label{divisorial N}
With notation as in Theorem~\ref{cic arch}, assume $S$ is not completely integrally closed and let $\theta \in S^{*} \setminus S$.
Then $\theta^{-1} S \cap S = N$.
\end{corollary}

\begin{proof}
By Theorem~\ref{cic arch}, $(N :_F N) = S^*$, so $\theta N \subseteq N$, hence $N \subseteq \theta^{-1} N \cap S \subseteq 
 \theta^{-1} S \cap S$.
Since also $\theta \notin S$, we have  $S \not\subseteq \theta^{-1} S$, hence $\theta^{-1} S \cap S \subsetneq S$.
Therefore, since $N$ is the maximal ideal of $S$, we have $\theta^{-1} S \cap S = N$.
\end{proof}

\begin{remark} \label{not coherent} {\em McAdam \cite{McAdam} defines  an integral domain 
$ A$  to be a {\it finite conductor domain}  if for elements $a, b$ in the field of fractions 
of $A$, the $A$-module $aA \cap bA$ is finitely generated.   A ring  is said to be 
{\it coherent} if every finitely generated ideal  is finitely presented. 
Chase \cite[Theorem~2.2]{Chase} proves that an integral domain $A$ is coherent if and only if the intersection of two finitely generated ideals of $A$ is finitely generated.  Thus a coherent 
domain is a finite conductor domain.  Examples of finite conductor domains that are not 
coherent are given by Glaz in  \cite[Example~4.4]{Glaz2} and by Olberding and Saydam 
in \cite[Prop. 3.7]{OS}. 
If  $S$ is archimedean but not completely integrally closed, then $S$ is not finite conductor and 
thus not 
 coherent. Indeed, if $S$ is archimedean and  coherent, then 
Corollary~\ref{divisorial N}
 implies that $N$ is a finitely generated ideal of $S$, which by Proposition~\ref{maximal ideal} implies that $N$ is a principal ideal. However,  Lemma~\ref{principal case} and Theorem~\ref{hull} then   imply that the Noetherian hull of $S$ is a fractional ideal of $S$, a contradiction to Theorem~\ref{cic}.}
\end{remark}

\begin{remark} {\em Following \cite[page~524]{Gil}, an integral domain  $A$ with field of fractions $K$ is a \emph{generalized Krull domain}  if 
there is a set $\mathcal{F}$ of rank $1$ valuation overrings of $A$  such that:
(i) $A = \bigcap_{V \in \mathcal{F}} V$; (ii) for  each $(V, {\ff M}_V) \in \mathcal{F}$, we have  $V = A_{{\ff M}_V \cap A}$; and (iii) 
$\mathcal{F}$ has finite character; that is, if $x \in K$ is nonzero, then $x$ is a nonunit in
only  finitely many valuation rings of $\mathcal{F}$. This class of rings has been studied by a number of authors; see for example \cite{Gri,Gri2,HO,PT,Pir,Rib}.
%
 In our setting,  when $S \subsetneq S^{*}$ the ring $S^{*}$ is  a generalized Krull domain whose defining family ${\mathcal F}$ consists of rank 1 valuation rings such that all but at most one member (namely, $W$) is a DVR.}


%
\end{remark}

In light of Theorem~\ref{cic arch}, which describes the complete integral closure $S^*$ of $S$ in the archimedean case, it remains to describe $S^*$ when $S$ is not archimedean. We do this in 
Theorem~\ref{cic}. 


\begin{theorem} \label{cic}
Assuming Setting~\ref{setting 2}, the following statements are equivalent. 

\begin{itemize}

\item[{\em (1)}] $S$ is not archimedean. 

\item[{\em (2)}] $T$ is a fractional ideal of $S$.

\item[{\em (3)}] $S^*=T$.

\item[{\em (4)}] The boundary valuation of $S$ has a nonzero nonmaximal prime ideal  that does not lie over $N$.

\item[{\em (5)}]  The ideal $\bigcap_{i>0}x^iS$ is a nonzero prime ideal of $S$.

\item[{\em (6)}] There  exists  $0 \ne y \in R$ such that for each $n>0$, $\ord_i(y) \geq n\cdot \ord_i(x)$ for all $i\gg 0$. 



\end{itemize}
\end{theorem}

\begin{proof} 
(1) $\Rightarrow$ (2) Since $S$ is not archimedean and $xS$ is $N$-primary, it follows that $\bigcap_{i>0}x^iS \ne 0$.  
By Theorem~\ref{hull}(1), $T = S[1/x]$, so that $0 \ne \bigcap_{i>0}x^iS = (S:_S T)$, and hence $T$ is a fractional ideal of $S$.   

(2) $\Rightarrow$ (3)  Since $T$ is a normal Noetherian domain, $T$ is completely integrally closed. Thus since $T$ is a fractional ideal of $S$, it follows that $T$ is the complete integral closure of $S$.  

(3) $\Rightarrow$ (4)  Since $T$ is the complete integral closure of $S$, $T$ is contained in every rank one valuation overring of $S$. 
For a  rank one valuation overring $U$ of $S$, it follows that $S \subsetneq T \subseteq U$.
Since the maximal ideal of $U$ lies over a nonzero prime ideal of $T$ and $N T = T$, we conclude that it lies over a nonzero nonmaximal prime ideal of $S$.
Statement (4) follows.

 
 (4) $\Rightarrow$ (5)  
 Since $\dim S  >1$, Theorem~\ref{hull}(1) implies that $T = S[1/x]$.  Let $V$ be  the  boundary 
 valuation ring for $S$.   By statement (3) there exists a nonzero prime ideal $Q$ of $V$ such that $Q \cap S$ is a nonmaximal prime ideal of $S$. Hence $x \not \in Q$  and  $Q \subseteq \bigcap_{i}x^iV$.  Since $V$ is a valuation ring, $\bigcap_{i \ge 1}x^iV$ is a nonzero prime ideal of $V$.   Using the fact that $x$ is a unit in $T$, 
 we have  $P = \bigcap_{i\ge 1}x^iV  \cap T = \bigcap_{i\ge 1} x^iS$ is a nonzero prime ideal of $S$.

(5) $\Rightarrow$ (6) 
Let $0 \ne y \in (\bigcap_{i}x^iS) \cap R$, and let $n>0$.  Then $y/x^n \in S$. Since $S = \bigcup_{i>0}R_i$, where $\{R_i\}$ is the sequence of quadratic transforms determined by $S$, it follows that $\ord_i(y) - n \cdot \ord_i(x) = \ord_i(y/x^n) \geq 0$ for all but finitely many $i$. 
%

(6) $\Rightarrow$ (1) Let $y$ be as in (6), and let $n>0$. Then (6) implies that for all but finitely many $i$, $y/x^n \in V_i$. Let $V$ be the boundary valuation ring for $S$. Since $V$ is a limit point for the $V_i$ in the patch topology, $y/x^n \in V$.  Since the choice of $n$ was arbitrary, we have $y \in \bigcap_{n>0}x^nV$, so that $V$ is not archimedean.
By Theorem~\ref{hull}(1), $T = S[1/x]$, and   by Theorem~\ref{flat}, $S = V \cap T$. 
Thus $y/x^n \in S$ for all $n>0$, and hence $y \in \bigcap_{n>0}x^nS$, which shows that  $S$ is not archimedean. 
\end{proof}

\begin{corollary} Assume Setting~\ref{setting 2}. Then there is a prime ideal $P$ of $S$ such that $S^* = (P:_F P)$. The ideal $P$ is maximal if and only if $S$ is archimedean. 
\end{corollary} 

\begin{proof} If $S$ is archimedean, then by Theorem~\ref{cic arch}, $S^*=(N:_FN)$. If $S$ is not archimedean, then  by Theorem~\ref{cic}, $P := \bigcap_{n>0}x^nS$ is a nonmaximal prime ideal of $S$. Also by Theorem~\ref{cic}, $S^* = T$. Since $T = S[1/x]$, it follows that $S^* = T \subseteq (P:_FP)$.  Moreover, since $T$ is completely integrally closed and $P$ is an ideal of $T$, we have $S^* = T = (P:_FP)$.
\end{proof}

\section{Shannon's examples}

Two examples \cite[Examples 4.7 and 4.17]{Sha} of David Shannon   motivated our work in this paper.   In Examples~\ref{Shannon example 4.7} and~\ref{Shannon example} we present details of these examples and their relation to concepts 
developed in this paper.   
The first, Example~\ref{Shannon example 4.7}, involves a nonarchimedean Shannon extension that is not a valuation ring, while the second, Example~\ref{Shannon example}, deals with an archimedean Shannon extension that is not a valuation ring.  In this section we make use
of the following elementary lemma.

\begin{lemma}\label{principal maximal lemma} {\em (cf.~\cite[Theorem 2.4]{Fon} and   \cite[Exercise 1.5, p.~7]{Kap})}
Let $A$ be a local domain with principal maximal ideal $\m = x A$ and let $\p = \bigcap_{n \ge 0} \m^n A$.
\begin{itemize}
\item[{\em (1)}]  $\p = x\p$ is a prime ideal, and every prime ideal properly contained in $\m$ is contained in $\p$.
\item[{\em (2)}]  $A$ is a valuation domain if and only if $A_{\p}$ is a valuation domain.
\end{itemize}
\end{lemma}

Example~\ref{Shannon example 4.7}   is based on \cite[Example 4.7]{Sha} of Shannon.  


\begin{example} {\em   \label{Shannon example 4.7}   Let $(R, \m)$ be a 3-dimensional regular local ring with 
$\m = (x,y,z)R$. Let $U$ be a  valuation ring  that birationally dominates $R$ such that, 
with $u$ the valuation of $U$,  we have 
 $nu(x) < u(y)$  and $nu(x) < u(z)$ for each positive integer $n$,  
 that is,  $u(x)$    
  is infinitely smaller than both $u(y)$ and $u(z)$.    Let $\{R_i\}_{i=0}^\infty$ be the sequence 
  of local quadratic transforms of $R$ along $U$.  The maximal ideal $\m_i$  of $R_i$ is 
  $\m_i = (x,  \frac{y}{x^i},  \frac{z}{x^i})R_i$.  For each $i$,   we have $z/y \not \in R_i$ and $y/z \not \in R_i$. 
  Hence  $S = \bigcup_{i}R_i$ is not a valuation ring.  
 
 Since $y,z \in \bigcap_{i}x^iS$,   the ring $S$ is not archimedean.  By Theorem~\ref{cic}, the complete integral closure $S^*$ of $S$ is the Noetherian hull $T = S[1/x]$ of $S$.  
 Observe that  $xS$ is the maximal ideal of $S$, and so by Lemma~\ref{principal maximal lemma}, $P = \bigcap_{i>0}x^iS$ is the unique largest nonmaximal  prime ideal of $S$.  It follows that $T = S[1/x] = S_P$.  Since 
 $(y, z)R \subseteq P \cap R \subsetneq \m$, and there are no prime ideals strictly between $(y, z)R$ and $\m$,
 we conclude that $T = R_{(y,z)R}$.

Since $S$ has principal maximal ideal $xS$,  we have $P = PS_P$ as sets.  Hence 
there are no rings properly between $S$ and $T=S[1/x]$.   Since $U$ dominates 
$S$,  we have  $T \not \subseteq U$.  Therefore $S = U \cap T$.  
However, $U$ need not be the boundary valuation ring of $S$.  The boundary valuation ring is unique, 
and  there are many possibilities for $U$;   all we require of $U$ is that  $U$ birationally dominates $R$ and 
 its valuation  $u$  has the property  that $u(x)$ is infinitely smaller than both $u(y)$ and $u(z)$.}
 \end{example}


%

Example~\ref{Shannon example}   is based on \cite[Example 4.17]{Sha} of Shannon.
The calculation of the complete integral closure of the archimedean Shannon extension $S$ in this  example relies on the following theorem, which gives a criterion for  the complete integral closure of $S$ to be a simple 
ring extension of $S$.

\begin{theorem}\label{cic one gen}
Assume notation as in Theorem~\ref{cic arch}.
If there exists $\theta \in S^* \setminus S$  such that $S [\theta]_{N} = W$,   then
\begin{itemize}
\item[{\em (1)}] $S^* = S [\theta]$,
\item[{\em (2)}]  $\theta^{-1} S [\theta^{-1}]$ is a maximal ideal of $S [\theta^{-1}]$, and 
\item[{\em (3)}]  $V = S [\theta^{-1}]_{\theta^{-1} S [\theta^{-1}]}$.
\end{itemize}
\end{theorem}

\begin{proof}
To show that $S [\theta] = S^*$, it suffices by Corollary~\ref{cic subrings} to show $S [\theta] / N = S^* / N$.
Let $k = S / N$ and let $\overline{(-)}$ denote image modulo $N$.
Since $\theta \not \in S$, Corollary~\ref{same units} implies $\theta^{-1} \notin S$, so by Seidenberg's Lemma \cite[Theorem 7]{Sei} it follows that $S [\theta] / N = k [\overline{\theta}]$ is a polynomial ring in one variable over the field $k$.
Thus from $S [\theta]_{N} = W$, it follows by permutability of localization and residue class formation that $W / N W = k (\overline{\theta})$ is a simple transcendental field extension.
Thus $k [\overline{\theta}] \subseteq S^* / N \subseteq k (\overline{\theta})$, so $S^* / N$ is a localization of $k [\overline{\theta}]$.
By Corollary~\ref{same units}, the units of $S^* / N$ are the units of $k$, so we conclude that $S^* / N = k [\overline{\theta}]$.


Let $A = S [\theta^{-1}]$, so $A \subseteq V$.
Then $N A$ is a prime ideal of $A$, $A_{N A} = V_{N V}$, and again by Seidenberg's Lemma, $A / N A = k [\overline{\theta^{-1}}]$ is a polynomial ring in one variable over the field $k$.
Now by Theorem~\ref{cic arch}, $\theta \in (N:N)$. Thus $ N A \subseteq \theta^{-1}A$, so that  
 $(N, \theta^{-1}) A = \theta^{-1} A$ is a principal maximal ideal of $A$.
Moreover, since $A / N A \cong k [\theta^{-1}]$,  $N A$ is a prime ideal of $A$ just below $\theta^{-1} A$.

Let $\tilde{V} = A_{\theta^{-1} A}$, so $\tilde{V}$ is a local domain with principal maximal ideal.
Since $\theta \in S^{*} \setminus S = (W \setminus V) \cap T$, it follows that $\theta^{-1} \in {\ff M}_V$, so $V$ birationally dominates $\tilde{V}$.
Since $\tilde{V}$ is a local domain with principal maximal ideal $\theta^{-1}\tilde{V}$, it follows from Lemma~\ref{principal maximal lemma}(1) that 
 $N \tilde{V}$ is the unique prime ideal just below $\theta^{-1} \tilde{V}$
 and that $N\tilde{V} = N\tilde{V}_{N\tilde{V}}$.
Furthermore, $\tilde{V} [\theta] = \tilde{V}_{N \tilde{V}} = W$, so that $\tilde{V}_{N \tilde{V}}$ is a valuation ring. By Lemma~\ref{principal maximal lemma}(2), $\tilde{V}$ is valuation domain. 
Since $\tilde{V}$ is a valuation domain birationally dominated by $V$, we have $\tilde{V} = V$.
\end{proof}

\begin{ex} \label{Shannon example} {\em   Let $(R, \m)$ be a 3-dimensional regular local ring with 
$\m = (x,y,z)R$.  Let $u$ be a valuation of the quotient field of $R$ with the property that $u(x) = a$,
$u(y) = b$, $u(z) = c$ are rationally independent positive real numbers such that $c > a + b$.  
Let $\{(R_n, \m_n)\}_{n \ge 0}$ be the sequence of local quadratic transforms of $R = R_0$ along the
valuation ring  determined by $u$ and let $S = \bigcup_{n \ge 0}R_n$.   Shannon proves that $S$ 
is not a valuation ring. Indeed,  for each integer $i \ge 0$,   Shannon  proves that $\m_i = (x_i, y_i, z_i)R$,  where
$a_i = u(x_i), b_i = u(y_i),  c_i = u(z_i)$  are distinct rationally independent real numbers and 
$c_i \neq \min\{a_i, b_i, c_i\}$.  Thus the local quadratic transform from $R_n$ to $R_{n+1}$ is obtained
either
\begin{enumerate}
\item 
 by dividing by $x_n$,  in which case $x_{n+1} = x_n$,   $y_{n+1} = \frac{y_n}{x_n}$ and 
 $z_{n+1} = \frac{z_n}{x_n}$, or
 \item 
 by dividing by $y_n$,  in which case $x_{n+1} =  \frac{x_n}{y_n}$,   $y_{n+1} = y_n$ and 
 $z_{n+1} = \frac{z_n}{y_n}$.
 \end{enumerate}
 The valuation $u$ defines a rank one valuation domain that birationally dominates $S$. 
 By varying the value of the real number $u(z) = c$,  subject only to the condition that $c > a + b$,  we conclude that 
 there exist infinitely many rank one valuation domains that birationally dominate $S$.  
 
For each $i$, the elements $x_i$ and $y_i$ each generate an $N$-primary ideal of $S$.
Consider the element $\theta = \frac{z}{xy} = \frac{z_i}{x_i y_i}$.
We show that $\theta \in S^* \setminus S$ and that $S^{*} = S [\theta]$.

Let $T$ be the Noetherian hull of $S$, and let $V$ be the boundary valuation ring for $S$.   
Since $x, y$ are units in $T$, $\theta \in T$.
For each $i \ge 0$, it follows that $\ord_{i} (\theta) = \ord_{i} (\frac{z_i}{x_i y_i}) = -1$.
Thus $\theta \notin V$, so $\theta \notin S$ and $\theta^{-1} \in V$.
By Proposition~\ref{maximal ideal}, for each element $f \in N$, it follows that $\lim_{n \rightarrow \infty} \ord_n (f) = \infty$.
Thus $\ord_n (\theta f) > 0$ for $n \gg 0$, so $\theta f \in {\ff M}_V \cap T = N$.
Therefore $\theta N \subseteq N$, so $\theta \in (N :_{F} N) = S^{*}$.

 
 


Denote $A = S [\theta]$. Since $\theta, \theta^{-1} \not \in S$, Seidenberg's Lemma \cite[Theorem 7]{Sei} implies $N = N A$ is a prime ideal of $A$ and  $A / N \cong k [\overline{\theta}]$ is a polynomial ring in one variable over $k = S / N$. In particular, $N$ is a nonmaximal prime ideal of $A$.  Therefore, since by Corollary~\ref{d-1} there are no rank 3 valuation rings between $A$ and its field of fractions, $\dim A =2$ and  $\dim A_N = 1$.  
Moreover, 
since $A \subseteq S^*$ and, by Corollary~\ref{cic N}, $N$ is the center of $W$ on $S^*$, $N $ is the center of $W$ on $A$.
Furthermore, since $R_i [\theta]$ is integrally closed for each $i \ge 0$, it follows that $A$ is integrally closed.

We show that $A_{N} = W$.
The ring $A_{N}$ is an integrally closed dimension $1$ local domain that birationally dominates $R$ and has residual transcendence degree $1$ over $S$.  The valuation ring $U$ has rational rank three and hence is residually algebraic over $R$ \cite[Theorem 1]{Abh}, so it follows that $S$ is residually algebraic over $R$ also. Therefore, $A_N$ has residual transcendence degree $1$ over $R$.  
If $A_{N}$ is not a valuation domain, then it is birationally dominated by a valuation domain $B$ that has positive residual transcendence degree over $A_{N}$ \cite[Theorem~10, p.~19]{ZS}. Therefore, since $R$ has dimension 3 and $A_N$ has residual transcendence degree $1$ over $R$, it must be that $B$ has residual transcendence degree $2$ over $R$; cf.~\cite[Theorem~1]{Abh}.  
This implies $B$ is a prime divisor of $R$ that dominates $S$. 
However, a Shannon extension of $R$ that is birationally dominated by a prime divisor of $R$ is necessarily equal to $R_i$ for one of the local quadratic transforms along $S$ \cite[Proposition 4]{Abh}, so we obtain a contradiction to the fact that in our case $\{R_i\}$ is an infinite sequence.
Thus $A_N$ is a valuation domain that is birationally dominated by $W$, which forces $S[\theta]_N = A_{N} = W$.
Thus  by Theorem~\ref{cic one gen}, we have $S^* = S [\theta]$ and $S [\theta^{-1}]_{\theta^{-1} S [\theta^{-1}]} = V$

 Finally, we note that the rank 1 valuation ring $U$ of $u$ along which $S$ was defined is not the rank 1 valuation overring $W$ of the boundary valuation ring $V$ of $S$, simply because $U$ has rational rank 3 and no such valuation overring of a 3-dimensional regular local ring can properly contain a valuation ring that contains $R$ (cf.~\cite[Theorem 1]{Abh}.) }

\end{ex} 

\begin{remark}  
{\em Example~\ref{Shannon example} exhibits an archimedean Shannon extension that is neither completely integrally closed nor a valuation ring, and whose boundary valuation ring has rank $2$. 
We prove in \cite{HLOST} that  there exist 
examples of archimedian  Shannon extensions $S$ that are completely integrally closed and whose boundary valuation $V$ has rank 1 yet $S$ is not a valuation ring; i.e., $S \subsetneq V$.  }
\end{remark}

\section{When a Shannon extension is a valuation ring}

If $R$ is a regular local ring of dimension $2$, then a valuation ring $V$ of $R$ that birationally dominates $R$ is the union of  the sequence of local quadratic transforms of $R$ along $V$ \cite[Lemma 12]{Abh}. Moreover  $V$ 
 is  either zero-dimensional\footnote{The {\it dimension} of a valuation ring $V$ that birationally dominates $R$ is the transcendence degree of the residue field of $V$ over the residue field of $R$; cf.~\cite[p.~34]{ZS}.}
 or a prime divisor of $R$ \cite[Theorem 1]{Abh}. Since 
 the sequence of local quadratic transforms along a prime divisor is finite \cite[Proposition 4]{Abh}, it follows that the Shannon extensions of the two-dimensional regular local ring $R$ are  precisely the zero-dimensional valuation rings that birationally dominate $R$.  
  In higher dimensions, while a Shannon extension need not be a valuation ring, and a zero-dimensional valuation overring need not be a Shannon extension, much is known about when these extensions are valuation rings; cf.~\cite{Gra, GMR, GMR2,GR, Gra2, HKT2, Sha}. 
In this section we give additional characterizations of when a Shannon extension is a valuation ring and recover some of the  previously known characterizations from a different point of view.

\begin{theorem}  \label{valuation criterion} The following are equivalent for a Shannon extension $S$ of a regular local ring $R$.

\begin{itemize}

\item[{\em (1)}] $S$ is a valuation ring.


\item[{\em (2)}] $S$ has only finitely many height 1 prime ideals.


\item[{\em (3)}] Either {\em (a)} $\dim S = 1$ or {\em (b)} $\dim S = 2$ and the boundary valuation ring $V$ of  $S$ has  value group  ${\mathbb{Z}} \oplus G $, where $G$ is a subgroup of ${\mathbb{Q}}$ and the direct sum is ordered lexicographically. 


\end{itemize}

\end{theorem}

\begin{proof} 
%
We use in the proof that by Theorem~\ref{flat} we have $S = V \cap T$, where $V$ is the boundary valuation ring of $S$ and $T$ is the Noetherian hull of $S$.  In particular, 
there is  $x \in S$ such that $xS$ is primary for the maximal ideal $N$ of $S$ and $T = S[1/x]$. 

(1) $\Leftrightarrow$ (2)  It is clear that (1) implies (2) since the ideals of a valuation ring are totally ordered by inclusion. Conversely, suppose that $S$ has only finitely many height 1 prime ideals. 
  If $P$ is a nonmaximal prime ideal of $S$ such that $P$ has height $>1$,  then since $S_P$ is a localization of the Noetherian ring $T$, there exist infinitely many height 1 prime ideals of $S$ that are contained in $P$, contrary to (2). Therefore, every nonmaximal prime ideal $P$ of $S$ has height 1, and $T =S[1/x]$ is the intersection of the rings $S_P$, where $P$ ranges over the height 1 prime ideals of $S$. By assumption there are only finitely many such prime ideals $P$. Moreover,   since $T$ is an integrally closed Noetherian domain, $S_P$ is a DVR for each each height 1 prime ideal $P$ of $S$. Therefore, since $S = V \cap T$,  $S$ is an intersection of $V$ and finitely many DVRs. Since $S$ is local, this implies that $S$ is a valuation domain \cite[(11.11)]{Nag}.

(1) $\Rightarrow$ (3) 
 Suppose that $S$ is a valuation ring. If $S$ is a DVR, the claim is clear, so suppose that $S$ is not a DVR. 
 As an overring of the valuation ring $S$, $T$ is a valuation ring, and hence the Noetherian ring $T$ is either a DVR or the quotient field of $S$. If $\dim S > 1$, then necessarily $T$ is a DVR, and since every nonmaximal prime ideal of $S$ survives in $T=S[1/x]$, this forces $\dim S  = 2$.  Furthermore, by \cite[Proposition 14]{Gra}, $S/P$ has value group isomorphic to a subgroup of ${\mathbb{Q}}$.  

(3) $\Rightarrow$ (1)  If $\dim S  = 1$, then $T=S[1/x]$ is the quotient field of $S$, so that $S = V \cap T = V$, and hence $S$ is a valuation ring. Suppose that $\dim S  = 2$ and the value group of $V$ is ${\mathbb{Z}} \oplus G$, where $G$ is a subgroup of ${\mathbb{Q}}$.  If the nonzero nonmaximal prime ideal $P$ of $V$ lies over the maximal ideal $N$ of $S$, then $S$ is dominated by a DVR (namely, the localization of $V$ at $P$), but then by Corollary~\ref{DVR case}, $S$ is a DVR, contrary to $\dim S = 2$. Thus $P \cap S$ is a height 1 prime ideal of $S$ and hence $S_{P \cap S}$ is a localization of $T=S[1/x]$. Since by Theorem~\ref{hull}, $T$ is a localization of some $R_i$, it follows that $S_{P \cap S}$ is a localization of $R_i$ at a height 1 prime. In particular, $S_{P \cap S} $ is a DVR, which forces $S_{P \cap S} = V_P$.  Since $V \subseteq S_{P \cap S}$ and $V/P$ has rational value group with $V$  irredundant in the intersection $S = V \cap T$, it follows that $V$ is a localization of $S$ \cite[Lemma 3.1]{OIrr}. Since $V$ dominates $S$, this forces $S = V$, which verifies (1).  
\end{proof} 

\begin{remark}
{\em If the Shannon extension $S$ of $R$ is a valuation ring with $\dim S  = 2$, then by Theorem~\ref{valuation criterion}, the value group of $S$ has rational rank $2$.  
There is no such bound on the rational rank of a valuation ring obtainable as a Shannon extension $S$ when $\dim S = 1$.
 Granja \cite[Proposition 16]{Gra} has shown that if $R$ is a regular local  ring of dimension $d \geq 2$, then there exists a Shannon extension of $R$ that is a valuation ring and whose corresponding valuation has rational rank $d$.  }
\end{remark}

\begin{corollary} \label{dim 2 principal}  Let $S$ be a  Shannon extension of the regular local ring $R$. Then $S$ is valuation domain with discrete value group if and only if 
 $\dim S \leq 2$ and $S$ has a principal maximal ideal. \end{corollary}  

\begin{proof}  If $S$ is a valuation ring with discrete value group, then by Theorem~\ref{valuation criterion}, $\dim S \leq 2$, and since the value group of $S$ is discrete, $S$ has a principal maximal ideal. Conversely, suppose that $\dim S \leq 2$ and $S$ has a principal maximal ideal. 
If $\dim S = 1$, then $S$ is necessarily a DVR, so suppose that $\dim S = 2$. 
With the notation of Lemma~\ref{principal case}, $S/P$ is a DVR and $PS_P = P$. Moreover, by Theorem~\ref{flat}, there is $x \in S$ such that $xS$ is primary for the maximal ideal and $S[1/x]$ is a regular Noetherian domain. Since $S_P$ is a localization of $S[1/x]$ and $\dim S_P = 1$, it follows that $S_P$ is a DVR. This and the fact that $S/P$ is a DVR and $PS_P = P$ imply 
 that $S$ is a valuation domain with discrete value group.    
\end{proof}

\begin{corollary} \label{Abhyankar cor}  {\em (Abhyankar \cite[Lemma 12]{Abh})} If $R$ is a regular local ring with $\dim R  = 2$, then the set of Shannon extensions of $R$ is precisely the set of zero-dimensional valuation overrings of $R$ that   dominate $R$.
\end{corollary}

\begin{proof} Let $V$ be a zero-dimensional valuation overring of $R$ that dominates $R$. Then $V$ determines an infinite  sequence of local quadratic transforms $\{R_i\}$ and hence there is a 
a Shannon extension $S$ of $R$ with $S \subseteq V$.  Since $\dim R = 2$, 
then $\dim R_i = 2$ for all $i \geq 0$.   If $\dim S = 1$, then $S$ is a valuation ring by Theorem~\ref{valuation criterion}. Suppose that $\dim S = 2$, and suppose by way of contradiction that   $S$ is not a valuation ring. Then there exists $u$ in the quotient field of $S$ such that $u,u^{-1} \not \in S$. Hence, with $N$ the maximal ideal of $S$, 
 Seidenberg's Lemma \cite[Theorem 7]{Sei} implies $NS[u]$ is a nonzero nonmaximal prime of $S[u]$. Therefore $S[u]$ is contained in a valuation ring $U$ with $\dim U  = 2$  such that the nonzero nonmaximal prime ideal $P$ of $U$ is centered on $NS[u]$. Since $\dim U=2$  and $U$ is an overring of the two-dimensional Noetherian domain $R$, the value group of $U$ is discrete \cite[Theorem 1]{Abh}, and hence $U_P$ is a DVR that dominates $S$. However, by Corollary~\ref{DVR case}, this implies $S$ is a DVR, contrary to assumption. Thus $S$ is a valuation ring, and since $V$ dominates $S$, $S = V$, which proves the corollary.  
\end{proof} 

In general, it is not enough that the maximal ideal of a Shannon extension $S$ is principal  to guarantee that $S$ is a valuation ring; see  Example~\ref{Shannon example 4.7}.   
However, with an additional assumption, $S$ must be a valuation ring:

\begin{corollary} \label{another DVR}  A Shannon extension $S$ of a regular local ring $R$ is a DVR if and only if 
the maximal ideal of $S$ is principal and $S$ is dominated by a rank 1 valuation ring. 
\end{corollary} 

\begin{proof}  Suppose that the maximal ideal of $S$ is principal and $S$ is dominated by a rank 1 valuation ring. The latter property implies that 
 $S$ is archimedean, and hence since $S$ has a principal maximal ideal,
 Lemma~\ref{principal case}(2) forces  $\dim S = 1$. 
Therefore,  by Corollary~\ref{dim 2 principal}, $S$ is a DVR.  The converse is clear.
 %
\end{proof}

Following Shannon  \cite{Sha}, the quadratic  sequence $\{R_i\}$ determined by the Shannon extension $S$   {\it switches strongly infinitely often} if $\epd(S/R)$ is empty, and following Granja \cite{Gra}, the sequence $\{R_i\}$  is  {\it height 1 directed} if $\epd(S/R)$ has exactly one element.  In Proposition~\ref{known} we show how to recover some results of Granja from our point of view. We use the notion of an essential prime divisor  from Definition~\ref{epd def}.

\begin{lemma} \label{equate} If $S$ is a Shannon extension of the regular local ring $R$ with $\dim S >1$, then $\epd(S/R) = \epd(S)$. 
\end{lemma}

\begin{proof} 
Let $V \in \epd(S/R)$. Then there exists a height 1 prime ideal $P_i$ of $R_i$ for some $i>0$ such that $S \subseteq V = (R_i)_{P_i}$.  Let $P$ be the contraction of the maximal ideal of $V$ to $S$. Then $S_P = V$, and hence $P$ has height 1 and $V \in \epd(S)$. Conversely, suppose that $P$ is a height 1 prime ideal of $S$. Since $\dim S>1$, we have $T \subseteq S_P$, where $T$ is as in Theorem~\ref{flat} and $T$ is a localization of $R_i$ for some $i$. Therefore, $S_P$ is a one-dimensional localization of $R_i$, which forces $(R_i)_{P \cap R_i} = S_P$ and $P \cap R_i$ to be a height 1 prime ideal of $R_i$.  Consequently, $S_P \in \epd(S/R)$.  
\end{proof}

 \begin{proposition} \label{known} {\em {(cf.~Granja \cite[Props.~7 and ~14, Thm.~13]{Gra})}}
 Assume Setting~\ref{setting 2}.  
 \begin{itemize}

 \item[{\em (1)}]  $\{R_i\}$ switches strongly infinitely often  if and only if $S$ is a valuation ring with $\dim S  = 1$.   
 \item[{\em (2)}] $\{R_i\}$ is height 1 directed if and only if $S$ is a valuation ring with $\dim S=  2$  and 
value group $G \oplus {\mathbb{Z}}$, where $G$ is a subgroup of ${\mathbb{Q}}$ and the sum is ordered lexicographically.

   \item[{\em (3)}]   $S$ is a valuation ring if and only if $\{R_i\}$ switches strongly infinitely often or $\{R_i\}$ is height 1 directed.

\end{itemize}
 \end{proposition} 
 
\begin{proof} (1) Suppose $|\epd(S/R)| = 0$.  Then by Lemma~\ref{equate} there does not exist a height $1$ prime ideal $P$ of $S$ such that $S_P$ is a DVR. Since $S[1/x]$ is a regular Noetherian domain and $x$ is primary for the maximal ideal of $S$  (with notation as in Setting~\ref{setting 2}), it follows that $\dim S = 1$. 
%
Hence by Theorem~\ref{valuation criterion},  $S$ is a valuation ring. 
Conversely, if $S$ is  a valuation ring with $\dim S = 1$, then  there exist no overrings properly between $S$ and its quotient field. Consequently, since $S$ is not a DVR (this possibility is ruled out by Setting~\ref{setting 1}(3)), $\epd(S,R)$ is empty. 

(2)  Suppose that $|\epd(S/R)| = 1$. Then by (1), $\dim S>1$, and hence by Lemma~\ref{equate}, $\epd(S/R) = \epd(S)$. Therefore, $S$ has only one height 1 prime ideal, and hence by Theorem~\ref{valuation criterion}, $S$ is a valuation ring with value group ${\mathbb{Z}} \oplus G$, where $G$ is a subgroup of ${\mathbb{Q}}$. The converse follows from Lemma~\ref{equate}. 

(3) In light of (1) and (2), it only needs to be observed that if $S$ is a valuation ring, then $|\epd(S/R)| = |\epd(S)| \leq 1$. 
\end{proof} 



 
 

{\it Acknowledgment.} The authors are grateful to  Angel Granja for a number of  helpful comments on an earlier version of this article. 
\medskip

{\noindent}

{\bf References}


\begin{thebibliography}{34}


\bibitem{Abh}
S.~Abhyankar, 
On the valuations centered in a local domain. 
Amer. J. Math. 78 (1956), 321--348. 

\bibitem{Abh2} 
S.~Abhyankar, Resolution of Singularities of Embedded Algebraic Surfaces, Springer-Verlag, 1998.




\bibitem{Chase} S.~Chase,  Direct products of modules, Trans. Amer. Math. Soc. 97 (1960), 457-473. 




\bibitem{Fon} 
M.~Fontana, 
Topologically defined classes of commutative rings.  Ann. Mat. Pura Appl. (4) 123 (1980), 331--355. 

\bibitem{FS} L.~Fuchs and L.~Salce,  {Modules over non-Noetherian domains,}   
Mathematical Surveys and Monographs   {84},   Amer. Math. Soc.  Providence, RI.

\bibitem{Gil} R.~Gilmer, {Multiplicative ideal theory,} Marcel Dekker, New York, 1972. 

\bibitem{Glaz2} S.~Glaz, Finite conductor rings. 
Proc. Amer. Math. Soc. 129 (2001), no. 10, 2833-2843. 


\bibitem{Gra}  A.~Granja, Valuations determined by quadratic transforms of a regular ring. 
J. Algebra 280 (2004), no. 2, 699--718. 

\bibitem{GMR} A.~Granja, M.~C.~Martinez and C.~Rodriguez, Valuations dominating regular local rings and proximity relations, J. Pure Appl. Algebra 209 (2007), no. 2, 371--382.

\bibitem{GMR2} A.~Granja, M.~C.~Martinez and C.~Rodriguez,
Valuations with preassigned proximity relations, J. Pure Appl. Algebra 212 (2008), 1347--1366. 

\bibitem{GR} A.~Granja and C.~Rodriguez,  Proximity relations for real rank one valuations dominating a local regular ring. Proceedings of the International Conference on Algebraic Geometry and Singularities  (Sevilla, 2001). Rev. Mat. Iberoamericana 19 (2003), no. 2, 393--412.


\bibitem{Gra2} A.~Granja and T.~S\'anchez-Giralda.
Valuations, equimultiplicity and normal flatness. 
J. Pure Appl. Algebra 213 (2009), no. 9, 1890--1900.

\bibitem{Gri} M.~Griffin, 
Rings of Krull type,
J. reine angew. Math. 229 (1968), 1--27.


\bibitem{Gri2} M.~Griffin, Families of finite character and essential valuations, Trans. Amer. Math. Soc. 130
(1968), 75--85.




\bibitem{HK} W.~Heinzer and M.-K.~Kim, The Rees valuations of complete ideals in a regular local ring, preprint.  

\bibitem{HKT} W.~Heinzer, M.-K.~Kim and M.~Toeniskoetter, 
Finitely supported $*$-simple complete ideals in a regular local ring
J. Algebra 401 (2014), 76--106.

\bibitem{HKT2} W.~Heinzer, M.-K.~Kim and M.~Toeniskoetter, Directed unions of local quadratic transforms of a regular local ring, preprint. 

\bibitem{HLOST} W.~Heinzer, K.~.A.~Loper, B.~Olberding, H.~Schoutens and M.~Toeniskoetter, Asymptotic properties of infinite sequences of local quadratic transforms, in preparation.  

\bibitem{HO} W.~Heinzer and J.~Ohm,
Defining families for integral domains of real finite character,
Canad. J. Math.  {\bf 24}  (1972), 1170--1177.

\bibitem{HR} W.~Heinzer and M.~Roitman, Well-centered overrings of an integral domain, J.~Algebra 272 (2004), 435--455.

\bibitem{Hir} H.~Hironaka, Resolution of singularities of an an algebraic variety of characteristic $0$ I and II, Ann. Math. 79 (1964), 109--326. 







\bibitem{Hoc} M.~Hochster, Prime ideal structure in commutative rings, Trans. Amer. Math. Soc. 142 (1969), 43--60.





\bibitem{Kap} I.~Kaplansky,  {Commutative Rings}, Allyn and Bacon, Boston, 1970.


\bibitem{Lip} J.~Lipman, On complete ideals in regular local rings. Algebraic geometry and commutative algebra, Vol. I, 203--231, Kinokuniya, Tokyo, 1988.




\bibitem{Mat} H.~Matsumura {Commutative Ring Theory}  Cambridge Univ. Press,  
Cambridge, 1986.



\bibitem{McAdam} S.~McAdam,  Two conductor theorems, J. Algebra 23(1972) 239-240.

\bibitem{Nag} M.~Nagata,  {Local Rings}, John Wiley, New York, 1962.


\bibitem{OIrr} B.~Olberding, Irredundant intersections of valuation overrings of two-dimensional Noetherian domains, J. Algebra
318 (2007) 834--855.

\bibitem{OS} 
B.~Olberding and S.~Saydam,  Ultraproducts of commutative rings, Commutative ring theory and applications (Fez, 2001), 369–386, Lecture Notes in Pure and Appl. Math., 231, Dekker, New York, 2003. 





\bibitem{PT} E.~Paran and M.~Temkin, 
Power series over generalized Krull domains. 
J. Algebra 323 (2010), no. 2, 546--550. 

\bibitem{Pir} E. Pirtle,  Families of valuations and semigroups of fractionary
ideal classes, Trans. Amer. Math. Soc. 144 (1969), 427--439.


\bibitem{Rib} P.~Ribenboim, Le th\'eor\`eme d'approximation pour les valuations de Krull,
Math. Zeit. 68 (1957/58), 1--18.


\bibitem{Ric} F.~Richman, 
Generalized quotient rings, 
Proc. Amer. Math. Soc. 16 (1965), 794--799. 

\bibitem{Sei} A.~Seidenberg, A note on the dimension theory of rings, Pacific J. Math. 3  (1953), 505--512.


\bibitem{Sha}
D.~Shannon, Monoidal transforms of regular local rings. Amer. J. Math. 95 (1973), 294--320. 

\bibitem{Zar38} O.~Zariski, Polynomial ideals defined by infinitely near base points, Amer. J. Math. (1938) 151--204.

\bibitem{Zar54} O.~Zariski, Applicazioni geometriche della teoria delle valutazioni.  
Univ. Roma. Ist. Naz. Alta Mat. Rend. Mat. e Appl. (5) 13, (1954). 51--88. 

\bibitem{ZS} O.~Zariski and P.~Samuel, {Commutative algebra.} Vol. II. Reprint of the 1960 edition. Graduate Texts in Mathematics, Vol. 29. Springer-Verlag, New York-Heidelberg, 1975.

\end{thebibliography}
\end{document}